\documentclass{amsart}
\usepackage[utf8]{inputenc}
\usepackage{setspace}
\usepackage[margin=1in]{geometry}
\usepackage{amsmath,amsfonts,amscd,amssymb,bm,amsthm, mathrsfs}
\usepackage{tikz-cd}
\usetikzlibrary{positioning,arrows,calc}

\usepackage{graphicx}
\usepackage{yfonts}

\usepackage[pdfencoding=auto,colorlinks,
            linkcolor=red,
            anchorcolor=blue,
            citecolor=green]{hyperref}
\usepackage{bookmark}

\usepackage[backend=bibtex]{biblatex}\bibliography{Pi1.bib}
\DeclareFieldFormat*{title}{#1}

\usepackage[all]{xy}
\usepackage{tikz}
\definecolor{myred}{RGB}{183,18,52}
\definecolor{myyellow}{RGB}{254,213,1}
\definecolor{myblue}{RGB}{0,80,198}
\definecolor{mygreen}{RGB}{0,155,72}
\newcommand{\mS}{\mathcal{S}}

\newcommand{\mJ}{\mathcal{J}}
\newcommand{\mT}{\mathcal{T}}

\newcommand{\mC}{\mathcal{C}}

\newcommand{\mX}{\mathcal{X}}
\newcommand{\mY}{\mathcal{Y}}
\newcommand{\mZ}{\mathcal{Z}}

\newcommand{\mG}{\mathcal{G}}

\newcommand{\mF}{\mathcal{F}}
\newcommand{\mE}{\mathcal{E}}

\newcommand{\mB}{\mathcal{B}}
\newcommand{\mD}{\mathcal{D}}
\newcommand{\aA}{\mathbb{A}}

\newcommand{\DD}{\mathbb{D}}
\newcommand{\EE}{\mathbb{E}}
\newcommand{\ZZ}{\mathbb{Z}}

\newcommand{\CC}{\mathbb{C}}
\newcommand{\mfC}{\mathfrak{C}}
\newcommand{\mfD}{\mathfrak{D}}

\newcommand{\msD}{\mathscr{D}}
\newcommand{\QQ}{\mathbb{Q}}
\newcommand{\RR}{\mathbb{R}}
\newcommand{\TT}{\mathbb{T}}

\newcommand{\Diff}{\mbox{Diff}}

\newcommand{\w}{\omega}

\newcommand{\cod}{\mbox{cod}}

\newtheorem{thm}{Theorem}[section]
\newtheorem{dfn}[thm]{Definition}
\newtheorem{cor}[thm]{Corollary}
\newtheorem{lma}[thm]{Lemma}
\newtheorem{prp}[thm]{Proposition}

\newtheorem{rmk}[thm]{Remark}

\newtheorem{cnd}{Condition}
\begin{document}
\title[Symplectic $(-2)$-sphere and the Symplectomorphism group]{Symplectic $(-2)$-spheres and the symplectomorphism group of small rational 4-manifolds}
\author{ Jun Li, Tian-Jun Li}
\address{Department  of Mathematics\\  University of Michigan\\ Ann Arbor, MI 48109}
\email{lijungeo@umich.edu}
\address{School  of Mathematics\\  University of Minnesota\\ Minneapolis, MN 55455}
\email{tjli@math.umn.edu}
\date{\today}
\begin{abstract}
 Let $(X,\w)$ be a symplectic rational surface. We study the space of tamed almost complex structures $\mJ_{\w}$ using a fine decomposition via smooth rational curves and a relative version of the infinite dimensional  Alexander-Pontrjagin duality.  This decomposition provides new understandings of both the variation  and stability of the symplectomorphism group
 $Symp(X,\w)$ when deforming $\w$.
 In particular, we compute  the rank of
 $\pi_1(Symp(X,\w))$ with  $\chi(X)\leq 7$
 in terms of   the number $N_{\w}$ of $(-2)$-symplectic sphere classes.
 \end{abstract}
\maketitle
\setcounter{tocdepth}{2}

\section{Introduction}\label{Intro}
Let $(X,\omega)$ be a closed simply connected symplectic manifold. The symplectomorphism group with the standard $C^{\infty}$-topology,  denoted by $Symp(X,\omega)$, is an infinite dimensional Fr\'echet Lie group.
Let $\mJ_{\omega}$ be the contractible Fr\'echet manifold of $\omega$-tamed almost complex structures.  Finding  a suitable  decomposition  of $\mJ_\w$  invariant under the natural action of $Symp(X, \omega)$ or its Torelli subgroup has proved useful in probing the topology of $Symp(X, \omega)$ when $(X, \omega)$ is a symplectic rational surface (\cite{Gro85,Abr98,AM00,LP04, AP13}, etc).  However, the analysis is usually hard even if $X$ is relatively simple (\cite{Anj02, AP13}).

Among homotopy groups of $Symp(X, \w)$, $\pi_0(Symp(X,\w))$ and $\pi_1(Symp(X,\w))$ are the most interesting ones.  $\pi_0(Symp(X,\w))$ is often called the {symplectic mapping class group}.
$\pi_1(Symp(X,\w))$ is tied to the Hofer geometry of $Symp(X,\w)$ (cf. \cite{Pol02}) and quantum cohomology (cf.\cite{Sei97}).

Note that $Symp(X,\omega)=  Symp_h(X,\w) \rtimes G(X,\omega)$,  where $Symp_h(X,\w)$ is  the homologically trivial part of $Symp(X,\omega)$  (the Torelli part), and  $G(X,\omega)$ is the image of the induced homomorphosm from $Symp(X,\omega)$ to $Aut[H^2(X,\ZZ)]$.   We also have $\pi_0 (Symp(X,\omega))=  \pi_0 (Symp_h(X,\w)) \rtimes G(X,\omega),$ and $\pi_k( Symp(X,\omega))=  \pi_k (Symp_h(X,\w)), \forall k\geq 1.$
It is shown in  \cite{LW12, She10}  that $G(X,\omega)$ is a reflection group generated by Dehn twists along Lagrangian spheres.
In fact, for a symplectic  rational surface $(X, \w)$  with $ \chi(X)\leq 11,$  we observe in
Lemma \ref{roots} and  section \ref{chileq5} that the collection of Lagrangian sphere classes form a root system, denoted  by $\Gamma_L(X_k,\w)$, and $G(X,\omega)$ is just the Weyl group of $\Gamma_L(X_k,\w)$.
Further, when  $\chi(X)\leq 7$, $Symp_h(X,\omega)$ is always path connected (\cite{LLW15}). Hence we know $\pi_0(Symp(X,\omega))=G(X,\omega)$ is just the Weyl group of the root system $\Gamma_L(X_k,\w)$.

We aim to  relate  $\pi_0(Symp(X,\w))$ and $\pi_1(Symp(X,\w))$ of a rational surface $X$, the codimension $0$ and $2$ pieces of a new decomposition of $\mJ_{\w}$,  and self-intersection $(-2)$-symplectic spheres  in symplectic rational surfaces.  In this paper, we focus on small rational surfaces with $\chi(X)\leq 7$.  In our upcoming papers \cite{LLW16} and \cite{LLW3} we will  extend our program to symplectic rational $4$-manifolds with larger Euler numbers.

\subsection{A fine decomposition of \texorpdfstring{$\mJ_{\omega}$}{lg} via symplectic spheres}
When $(X,\omega)$ is a  symplectic $4$-manifold, we introduce in Section 2 the  following decomposition of $\mJ_{\omega}$ via embedded $\w$-symplectic spheres of  self-intersection at most $-2$.

Let $\mathcal S_{\omega}$ denote the set of homology classes  of
  embedded $\omega$-symplectic spheres and $K_{\w}$ the symplectic canonical class. For any $A\in \mathcal S_{\omega}$, by the adjunction formula,
  \begin{equation}\label{AF}
K_{\omega}\cdot A=-A\cdot A -2.
\end{equation}
For any integer $q$, let $$\mathcal S_{\omega}^{\geq q},  \quad \mathcal S_{\omega}^{>q}, \quad   \mathcal S_{\omega}^{q}, \quad  \mathcal S_{\omega}^{\leq q},\quad  \mathcal S_{\omega}^{< q}$$
 be the subsets of $\mathcal S_{\omega}$ consisting of classes with square $\geq q, >q, =q, \leq q, <q$ respectively.
For each $A\in  \mathcal S_{\omega}^{<0}$ we associate the integer
$$ {\cod_A}=2(-A\cdot A-1).$$

Then we define the prime subset $\mJ_{\mC}$ labelled by a set $\mC\subset \mathcal S_{\omega}^{\leq-2}$ as following:

 \begin{dfn} \label{fine decomposition}
  A subset  $\mC\subset \mathcal S_{\omega}^{\leq-2}$ is called admissible if
  $$\mC=\{A_1,\cdots, A_i,\cdots  |\,\, A_i\cdot A_j \geq 0, \quad \forall  i\neq j\}.$$ Given an admissible subset $\mC$, we  define the real codimension
  of the label set ${\mC}$ as  $$\cod({\mC})= \sum_{A_i\in \mC} \cod_{A_i}=\sum_{A_i\in \mC}  2(-A_i\cdot A_i-1).$$ Define the {\bf prime subset}
  $$\mJ_{\mC}:=\{ J\in \mJ_{\omega}|  A\in \mathcal S_{\w}^{\leq-2} \hbox{  has an embedded $J$-holomorphic representative if and only if } A\in \mC\}.$$
 The prime subset  $\mJ_{\emptyset}$   is generally denoted by  $\mJ_{open}$. And if $\mC=\{A\}$ contains only one class $A$, we will use $\mJ_A$ for $\mJ_{\{A\}}$.
 \end{dfn}

  Notice that these prime subsets are disjoint and
 we have the decomposition $\mJ_{\omega} =\amalg_{\mC} J_{\mC}.$
We define a filtration according to the codimension of these prime subsets:
 $$
 \cdots \subset \mX_{2n+2}(= \mX_{2n+1}) \subset
 \mX_{2n}(=\mX_{2n-1}) \subset \ldots \subset \mX_2 (= \mX_1) \subset \mX_0 = \mJ_{\omega},
$$
  where $\mX_{j}:= \amalg_{\cod({\mC} )\geq j} \mJ_{\mC}$ is the union of all  prime subsets having codimension no less than $j$.

  To ensure nice properties of the decomposition into prime subsets we introduce the following condition for  $(X,\omega)$.

 \begin{cnd}\label{negsphemb} Let $(X, \w)$ be a symplectic $4-$manifold.
Suppose $A$ is a homology class in $H_2(X;\ZZ)$ with $A\cdot A<0$. Whenever
 $A$ is represented by  a simple $J$-holomorphic map
 $u:\CC P^1 \rightarrow (X, J)$ for some tamed $J$,  then $u$ is an embedding.
 \end{cnd}

 We note  that by \cite{Zha17}, Condition \ref{negsphemb} holds for symplectic rational surfaces with Euler number no larger than 11
  (Lemma \ref{rat1}).
We show that each $\mJ_{\mC}$ is  a submanifold with real codimension $\cod({\mC})$ under  Condition \ref{negsphemb}.
In Section 3, we prove  the following result  on $\mX_0, \mX_2,$ and $\mX_4$, which suffices for applications in this paper:

\begin{thm}\label{level 2}
For a symplectic rational surface with  $\chi(X)\leq 8$,
$\mX_4=\coprod_{\cod(\mC)\geq4} \mJ_{\mC}$ and $\mX_2=\coprod_{\cod(\mC)\geq2} \mJ_{\mC}$
 are closed subsets  in $\mX_0=\mJ_{\w}$. Consequently, $\mX_0 -\mX_4$ is an open  submanifold of $\mJ_{\omega}$ and $\mX_2-\mX_4$ is a closed  submanifold of  $\mX_0 -\mX_4$.





\end{thm}

  In \cite{L1} the first author will further prove the filtration is a stratification at every level for a symplectic rational surface with $\chi(X)\leq 8$.

 Inspired  by \cite{Abr98}, we apply a (relative) version of the Alexander-Pontrjagin duality in \cite{Eells61} to get the following computation of  $H_1(\mJ_{open};\ZZ)$.

\begin{cor}\label{Jopen}
For  a symplectic rational surface with  $\chi(X)\leq 8$ and any Abelian group $G$,
$H^1(\mJ_{open}; G)= \oplus_{A_i \in \mathcal S_{\omega}^{-2}} H^0(\mJ_{A_i}; G)$.

If we further assume that $\chi(X)\leq 7$, then  $\mJ_{A_i}$ is path connected  for  each  $ {A_i \in \mathcal S_{\omega}^{-2}}  $, and
$H_1(\mJ_{open}; \ZZ)=\ZZ^{N_{\w}}$, where $N_{\omega}$ is the cardinality of $\mathcal S_{\omega}^{-2}$.
\end{cor}




We apply Corollary \ref{Jopen} to study the topology of $Symp_h(X,\omega)$, where $X$ is a rational surface with $\chi(X)\leq 7$.
Notice that there is an action  of  $Symp(X, \w)$ on  $\mJ_{\w}$ which preserves the filtration. Moreover,  the subgroup  $Symp_h(X, \w)$ preserves  each prime subset
 since   $Symp_h(X, \w)$  preserves the homology classes of $J$-holomorphic curves. This action will be used in \cite{LLW16} and \cite{LLW3} to study the topology of $Symp_h(X,\omega)$  for a rational surface $X$ with $\chi(X)\geq 8$.

\subsection{Application to the symplectomorphism group}


For a symplectic rational surface $(X, \w)$ with $5\leq \chi(X)\leq 8$, the following diagram of  homotopy fibrations, formulated in \cite{Eva11} (in the monotone case, i.e. $[\w]=\lambda c_1(X,\w)$ for some $\lambda>0$ ) and adapted in \cite{LLW15}  for a general $\omega$, relates $\mJ_{open}$ and  $Symp_h(X,\omega)$:
\begin{equation} \label{summary}
\begin{CD}
Symp_c(U)\sim Stab^1(D) @>>> Stab^0(D) @>>> Stab(D) @>>> Symp_h(X, \omega) \\
@. @VVV @VVV @VVV \\
@. \mG(D) @. Symp(D) @. \msD_0 \sim \mJ_{open}
\end{CD}
\end{equation}
 Each term above is a topological group except the term $\msD_0 \sim \mJ_{open}$.
 Here $\sim$ means weak homotopy equivalence. We will  recall  each term in Section 4 and here we only explain the right tail of  diagram \eqref{summary}:
 \begin{equation}\label{right}
Stab(D)\to Symp_h(X, \omega)\to \msD_0 \sim \mJ_{open}.
\end{equation}

Here $D=\cup D_i$ is
a suitable exceptional divisor, which is a configuration of symplectic  spheres as in \cite{Eva11,LLW15}, with $[D_i]\in \mathcal S_{\omega}^{-1}$.  The space  $\msD_0$ of such configurations whose components intersect symplectic orthogonally is weakly  homotopic to
$\mJ_{open},$ and it admits a  transitive action of $Symp_h(X, \omega)$. Therefore we have the above homotopy fibration \eqref{right} with   $Stab(D)$ as the stabilizer of the transitive action.  Moreover,  the (weak) homotopy type of $ Stab(D)$ can often be explicitly computed using the terms of the other parts of diagram \eqref{summary}. Hence if we can further reveal the homotopy type of $\mJ_{open}$, which is very sensitive to the symplectic structure $\omega$, we may probe at least partially the homotopy type of $Symp_h(X, \omega)$ via the homotopy fibration \eqref{right}.

Following this route, the full homotopy type of  $Symp_h(X, \omega)$ in the monotone case  is determined in \cite{Eva11} when $\chi(X)=6,7,8$
(the smaller  $\chi$ cases follow from \cite{Gro85} and  \cite{Abr98,  LP04}). And we show  in \cite{LLW15} that $\pi_0(Symp_h)$ is trivial for $\chi(X)=7$  with a general $\omega$ (the smaller  $\chi$  cases have been dealt in  \cite{Abr98, AM00, LP04,AP13}). In addition, \cite{HPW13} treats similarly some  non-compact cases. In this paper we continue to follow this route and  systematically analyze the persistence and change of the topology of $Symp(X,\omega)$ under deformations of symplectic forms (such phenomena were also discussed  in \cite{Ruan93, Sei99} and \cite{McD08}).

\subsubsection{\texorpdfstring{ $\pi_1(Symp_h(X,\w))$}{Lg} and \texorpdfstring{$N_{\omega}$}{Lg}}

We are able to  relate the fundamental group  of $Symp_h(X,\w)$ with $N_{\w}$ for a rational surface $X$ with $\chi(X)\leq 7$.
On the one hand, for a rational surface $X$ with $\chi(X)\leq 7$,   we compute  $H_1(\mJ_{open};\ZZ)$ by counting $(-2)$-symplectic sphere classes  in Corollary \ref{Jopen}. On the other hand,  when  $5\leq \chi(X)\leq 7$,  we showed  in \cite{LLW15} that $Stab(D)$ is path connected for any $\w$.
Hence we have the following portion of the long exact sequence of fibration \eqref{right}:
\begin{equation}\label{sp1}
  \pi_1(Stab(D))\to \pi_1(Symp_h(X,\w))\to \pi_1(\msD_0) \to 1.
\end{equation}
Note that $\pi_1(Symp_h(X,\w))$ is Abelian since $Symp_h(X, \w)$ is a topological group.
Thus we can conclude  that  $ \pi_1(\msD_0)$ is Abelian since it is a quotient group of $\pi_1(Symp_h(X,\w))$ by \eqref{sp1}.
Since $\msD_0 \sim \mJ_{open}$, we have $\pi_1(\mJ_{open})=\pi_1(\msD_0)$ is also an Abelian group, hence isomorphic to
$H_1(\mJ_{open};\ZZ)$ by the Hurewicz theorem. Consequently,  $\pi_1(\msD_0)=\ZZ^{N_{\omega}}$.

We further examine the first map in \eqref{sp1}.
It turns out that, $\pi_1(Stab(D))$ is independent of $\w$ by Proposition \ref{stab} and is  isomorphic to $\pi_1(Symp_h(X,\omega_{mon}))$, where $\omega_{mon}$ is a monotone symplectic form. Moreover, Lemma \ref{Abel}  shows that $\pi_1(Stab(D))$  injects into
 $\pi_1(Symp_h(X,\w)).$ Since $\pi_1(\msD_0)= \ZZ^{N_{\omega}}$ is free Abelian,  $\pi_1(Symp_h(X,\w))$ is determined as follows, at least in the case $5\leq \chi(X)\leq 7$:

 \begin{thm}\label{sum}
 If $(X, \omega)$ is a symplectic rational surface with  $\chi(X)\leq 7$,
 \begin{equation}\label{sum formula}
 \pi_1(Symp_h(X,\omega))= \ZZ^{N_{\omega}} \oplus \pi_1(Symp_h(X,\omega_{mon})).
 \end{equation}
 \end{thm}

For the case $\chi(X)\leq 4$, we verify the relation \eqref{sum formula} directly from the known computations.

\subsubsection{Stability and variation along the normalized reduced symplectic cone}
Notice that diffeomorphic symplectic forms have homeomorphic symplectomorphism groups.
On the other hand,  for a rational surface $X$,   cohomologous symplectic forms are isotopic  (\cite{LM96}, \cite{LL01}). Therefore, up to homeomorphisms, $Symp(X, \omega)$  only depends on $[\omega]$. In particular, the homotopy type of $Symp(X, \omega)$ is determined by the point $[\w]$ in the symplectic cone of $X$. Recall that the symplectic cone of $X$ is the open subset of $H^2(X;\mathbb R)$ that consists of classes of symplectic forms.
The symplectic cone of $S^2\times S^2$ is just the first quadrant in $H^2(S^2\times S^2;\RR)$.
The symplectic cone of $X_k=\CC P^2  \# k{\overline {\CC P^2}}$ is characterized  (cf. \cite{Bir97,Bir01,LL01}) by
a set of
inequalities (infinitely many if $k\geq 2$).   However, after identifying $H^2(X_k;\mathbb R)$ as $\mathbb R^{k+1}$ via the canonical basis $\{H, E_1, \cdots, E_k\}$,
it is rather complicated to geometrically describe the symplectic cone in  $\mathbb R^{k+1}$.

Here is a useful observation.  It suffices to describe $Symp(X_k, \omega)$ when $[\w]$ is in a {\bf fundamental domain} of the symplectic cone under
 the homological action of $\Diff^+(X_k)$ and the scaling operation.

 For $X_k=\CC P^2  \# k{\overline {\CC P^2}}$,
the notion of a reduced class in Definition \ref{reduced} is useful for this purpose and  we use the normalization  $\w(H)=1$.
Hence,  for $X_k=\CC P^2  \# k{\overline {\CC P^2}}$,  we study the normalized reduced symplectic cone $P_k=P(X_k)$. Notice that we can identify each point in $P_k$  as a vector $(c_1,c_2,\cdots, c_k)\in \RR^k$, where $c_i$ denote the  $\w$-area of  $E_i$.
For $S^2\times S^2$, the normalized reduced symplectic cone can be similarly defined  and
$P(S^2\times S^2)$ is the interval $[1, \infty)\subset \mathbb R^1$.
When $\chi(X)\leq 11$, $P(X)$ is defined by finitely many (strict and non-strict) linear  inequalities in  $\RR^{\chi(X)-3}$,  hence a (generally, neither closed nor open) convex polytope of dimension $\chi(X)-3$.
 Notice that $P(X)$ is the disjoint union of its open faces of various dimensions, and notice that neighboring open faces have different dimensions.
 For instance, $P(S^2\times S^2)=[1, \infty)$ is the disjoint union of the unique one dimensional open face, $(1, \infty)$, and the unique zero open face, $\{1\}$.

Moreover,  for  $3 \leq k\leq 8$,  $P_k$ has a uniform description as   a cone with an open base in the
$c_1\cdots c_{k-1}$ hyperplane, with the point $M_k=(\frac{1}{3},\frac{1}{3},\cdots, \frac{1}{3})$  as the vertex and $k$ edges.   Notice that $M_k$ is the class of a monotone symplectic form on $X_k$. $P_k$ has a unique zero dimensional open face $M_k$, $k$ 1-dimensional open faces, each one being the interior of an edge,
and generally, $\binom{k}{p}$ p-dimensional open faces.
 It is interesting that the $k$ edges of $P_k$ have a  Lie theory interpretation as simple roots of a root system $R(X_k)$, well known to algebraic geometers.
We further observe that the root system $R(X_k), 3\leq k\leq 8,$ coincides with the Lagrangian root system  of $(X_k, \w_{mon})$, where $\w_{mon}$ is a monotone symplectic form  on $X_k$.

We observe that, when $\chi(X)\leq 7$,  both $\Gamma_L(X_k,\w)$ and $N_{\w}$ are stable on each open face of $P(X)$ and always  jump between neighboring open faces (necessarily of different dimensions) with the same ``amount of change".
This observation leads to the following  explicit stability and variation of $\pi_0 (Symp(X,\omega))$ and $\pi_1(Symp(X,\omega))$.

\begin{cor}\label{4Q}
For a  rational surface $X$ with  $\chi(X)\leq 7$, $\pi_0 (Symp(X,\omega))$ and $\pi_1(Symp(X,\omega))$ are stable on each open face of $P(X)$ and
vary between neighboring open faces. Moreover, the sum
$r^+[\pi_0 (Symp(X,\omega))]+  Rank [ \pi_1(Symp(X,\omega))]$ is  a  constant given by   $\frac{1}{2}(\chi(X)-2)(\chi(X)-3)$.

\end{cor}

Notice that in the above corollary, $\pi_1(Symp(X,\omega))$ is an Abelian group and
$Rank [ \pi_1(Symp(X,\omega))]$ denotes the  rank of its torsion-free part.  When $\chi(X)\leq 7$,  $\pi_0(Symp(X,\omega))$ is   the Weyl group  of the Lagrangian root system $\Gamma_L(X_k,\w)$,  and, by a slight abuse of notation,   $r^+[\pi_0 (Symp(X,\omega))]$ denotes  the number of positive roots of $\Gamma_L(X_k,\w)$.



\smallskip
  {\bf Acknowledgements:} We appreciate useful discussions with  Silvia Anjos, Martin Pinsonnault, Weiwei Wu, Weiyi Zhang. For $\CC P^2  \# 4{\overline {\CC P^2}}$, when $[\w]$ is in the  face $M_4OBC,$  Anjos-Eden in \cite{AE17} computed  the rational homotopy groups using toric method (see Remark \ref{anj}). Results in Section 3.1 overlap with results in Section 4.1 in  \cite{Zha17}, which are in a slightly different context. We thank the referees for their detailed and constructive comments, which greatly improved the exposition of this article.  The research is supported by NSF grant DMS-1611680.

\bigskip
{\bf List of some notations. }\\
$\bullet$ $Symp_h(X,\w)$: the  homological trivial part of symplectomorphism group of $(X,\w)$. \\
$\bullet$ $\mJ_{\w}$: the space of $w-$tamed almost complex structures.\\
$\bullet$ $\mC$, $\mD$ and $\mathfrak{D}$:
subsets of $ \mathcal S_{\omega}^{< 0}$.\\
$\bullet$ $\mX_{2n}$: the union of codimension no larger than $2n$ prime subsets of $\mJ_{\w}$.\\
$\bullet$ $\mJ^{\mfD}$: the set of almost complex structures $J$ where each class in $\mfD$ has an  embedded $J$-holomorphic curve representative.\\
$\bullet$ $N_{\w}$:  the number of $\omega$-symplectic $(-2)$-sphere classes.\\
$\bullet$   $N_{\w, L}$:  the number of  $\omega-$Lagrangian sphere classes up to  sign.\\
$\bullet$ $\Gamma_L(X,\w)$: the root system of Lagrangian sphere classes associated to a symplectic 4-manifold $(X,\w)$.\\
$\bullet$ $P(X)$: the normalized reduced symplectic cone of $X$, and $P_k:=P(X_k).$\\
$\bullet$ $M_k$: the (unique) monotone vertex of $P_k$ for $3\leq k \leq 8$.\\
$\bullet$ $r^+[\pi_0 (Symp(X,\omega))]$: the number of positive roots of the root system $\Gamma_L(X,\w)$.
\section{Decomposition of \texorpdfstring{$\mJ_{\w}$}{Lg} and the reduced symplectic cone when \texorpdfstring{$\chi\leq 11$}{Lg}}


In this section, we analyze the
 decomposition of $\mJ_{\omega}$ and describe the normalized reduced symplectic cone for symplectic rational surfaces with $\chi\leq 11$. We often identify the degree 2 homology with degree 2 cohomology using Poincar\'e duality. Also recall that throughout this section, all finite-dimensional symplectic manifolds are smooth, closed and connected.

\subsection {Decomposition of \texorpdfstring { $\mJ_{\omega}$}{Jw} }

\subsubsection{General facts on \texorpdfstring{$J$}{Lg}-holomorphic curves }
We review some general facts on $J$-holomorphic rational curves and symplectic spheres in symplectic 4-manifolds. The presentation is  similar to  \cite{Abr98} and \cite{AP13}.

Let $(X, \w)$ be a symplectic $4$-manifold and $J\in \mJ_{\w}$.
 A parametrized $J$-holomorphic  curve in $X$ is a $J$-holomorphic map $u:(\Sigma, j)\to (X, J)$, where $(\Sigma, j)$ is a smooth, connected Riemann surface.
We will always assume that $u$ is simple, i.e. it is non-constant and not a multiple covering. In this case, we say that $C=u(\Sigma)$ is an (unparameterized) $J$-holomorphic curve
and denote by $[C]$ the homology class. Notice that the pairing $\omega([C])$ is positive.



\begin{thm}
[Positivity of Intersection, \cite{Gro85, McD94,MW95}]\label{positive}
 For an almost complex  4-manifold $(X, J)$,
 two distinct simple  J-holomorphic  curves
 $C, C'$ have only finitely many intersection points. Each such point $p$ contributes
  $k_p\geq 1$  to the homological intersection number $[C]\cdot [C']$, and $k_p=1$ if and only if
$C$ and $C'$  meet transversally at $p$.
\end{thm}

\begin{thm}[Adjunction Inequality, \cite{McD90}]\label{adj ineq theorem}
Let $ (X, J)$ be an
almost complex 4-manifold with first Chern class $c_1(X, J)$ and $u: (\Sigma, j) \rightarrow X$ a simple $J-$holomorphic  curve.
Then the virtual genus of the image $C = u(\Sigma)$, defined as $g_v(C)= ([C]\cdot [C] -c_1(X, J)([C]))/2+1$, is a  integer no less than $g(\Sigma)$.
Moreover, $g_v(C)= g(\Sigma)$ if and only if $u$ is an embedding.
\end{thm}

We are primarily interested in rational curves.
A (parametrized) $J$-holomorphic rational curve in $X$ is a $J$-holomorphic map $u:\CC P^1\to (X, J)$.


\begin{thm}[Gromov Compactness Theorem for rational curves, \cite{Gro85}]\label{Gcpt}
 Let $(X, \omega)$
be a  symplectic 4-manifold, and let $J_n$ be a sequence in  $\mJ_{\w}$ which converges to $ J_{0} $ in the $C^{\infty}$-topology.
Suppose $u_n : \CC P^1 \rightarrow (X, J_n)$
 is a sequence of simple $J_n$-holomorphic rational curves such that $[u_n(\CC P^1)] = A \in H_2(X; \ZZ), A\ne0$.

 Then,  there is a subsequence of $\{u_n\}$, still denoted $\{u_n\}$,  which  either converges to a simple $J_0$-holomorphic rational curve in the class $A$ in the $C^{\infty}$-topology, or
 weakly converges to a stable $J$-holomorphic rational curve in the class $A$.

 The weak convergence of $\{u_n\}$  is described as follows.

\begin{itemize}

\item Up to a re-parametrization
of each $u_n$,  there  are finitely many  simple closed loops ${\gamma_i}$ in $\CC P^1$,
 and  a connected finite union of Riemann spheres $\Sigma_0 = \cup_{\alpha} \CC P^1_{\alpha}$
 which  is  obtained by
collapsing each of the simple closed curves ${\gamma_i}$ on $\CC P^1$ to a point.

 \item There is a continuous map $u : \Sigma_0 \rightarrow X$,
 called a {\bf stable rational curve} in the class $A$,  such that each $u|_{\CC P^1_{\alpha}}$ is a possibly multiply covered  $J_0$-holomorphic rational curve, i.e. $u|_{\CC P^1_{\alpha}}$ is the composition of a degree $m_{\alpha}$ covering
 $\CC P^1_{\alpha}\to \CC P^1$ and a simple $J_0$-holomorphic map $u'_{\alpha}:\CC P^1\to (X, J_0)$.
 Moreover,
  \begin{equation}\label{cusp curve equation}
 \sum_{\alpha} m_{\alpha} [u'_{\alpha}(\CC P^1_{\alpha})] = A.
 \end{equation}
Finally,  $\{u_n\}$ converges to $u$   in the complement of any fixed open neighborhood of $\cup_i \gamma_i$ in the $C^{\infty}$-topology.

\end{itemize}


\end{thm}

\subsubsection{General facts on symplectic spheres }

We next list various facts about representing a class in $\mS_{\w}$ by simple (or stable)  $J$-holomorphic rational curves.

For classes with self-intersection at least $-1$ we have the following well-known fact from the Gromov-Witten theory.
The following Proposition is Lemma 3.3 in \cite{LZ11}.

\begin{prp} \label{existence of curves}
Let $ (X, \omega)$ be a  symplectic 4-manifold and $A$ a class in $ \mathcal S_{\omega}^{-1}$ or $  \mathcal S_{\omega}^{\geq 0}$.   Then  there is
a simple or stable $J$-holomorphic rational curve in the class $A$ for any $J\in \mJ_{\w}$, and
by the adjunction inequality, any simple $J$-holomorphic curve in the class $A$ is embedded.

Moreover, if $\{B_i\}$ is any collection of  classes in  $ \mathcal S_{\omega}^{-1}$ or $  \mathcal S_{\omega}^{\geq 0}$,
then for a generic $J\in \mathcal J_{\omega}$ (in a subset of second category, see also Lemma \ref{open and -1}), there is an
 embedded $J$-holomorphic rational curve in each class $B_i$.
 Consequently, by the positivity of intersection, $B_i\cdot B_j\geq 0$ if $B_i\ne B_j$ are in $ \mathcal S_{\omega}^{-1}$ or $  \mathcal S_{\omega}^{\geq 0}$.
\end{prp}

On the other hand, for classes with negative self-intersection, we have the uniqueness by the positivity of intersection.

 \begin{lma} \label{uniquecurves}
Let $ (X, \omega)$ be a symplectic 4-manifold and $B$ a class in $ \mathcal S_{\omega}^{\leq -1}$.  If, for some $J\in \mJ_{\w}$,  there is a simple $J$-holomorphic  rational curve in the class $B$, then  there cannot be a distinct  simple $J$-holomorphic  curve or a stable $J$-holomorphic rational curve in $B$.
\end{lma}

 For classes with self-intersections at most $-2$ we have  the following fact  in
\cite{AP13} Appendix B.1:
\begin{prp}\label{stratum}
Let $(X,\omega)$ be a symplectic 4-manifold.
 Suppose $U_{\mC}\subset\mJ_{\omega}$ is a subset
characterized by the existence of a configuration of embedded
$J$-holomorphic  rational curves $C_{1}\cup C_{2}\cup\cdots\cup C_{N}$ of negative self-intersection with $\{ [C_{1}], [C_{2}],\cdots , [C_{N}]\}=\mC$.
Then $U_{\mC}$ is a co-oriented Fr\'echet submanifold
of $\mJ_{\omega}$ of (real) codimension $2N-2c_{1}([C_{1}]+\cdots+ [C_{N}])$.
\end{prp}

\begin{rmk}
Suppose $(X, \w)$ is a symplectic rational surface with $\chi(X)\leq 8$. The we will show in Proposition \ref{negative sphere}  that
the set $\mS_\w^{<0}$ is finite.  Hence   there are finite number of admissible sets and each admissible set $\mC$ is finite in this case.
It should follow from Section 4 in \cite{Zha17} that the finiteness of $\mS_\w^{<0}$ continues to hold when $\chi(X)\leq 11$.
\end{rmk}


\subsubsection {Condition \ref{negsphemb} and cusp curve decomposition}


We first recall  Condition  \ref{negsphemb} for a symplectic 4-manifold $(X,\omega)$:
Suppose $A$ is a  class in $H_2(X;\ZZ)$ with $A\cdot A<0$.
Whenever
 $A$ is represented by  a simple $J$-holomorphic map $u:\CC P^1 \rightarrow X$
  for some tamed $J$,  then $u$ is an embedding.

Note that, under this assumption, $\mathcal{S}_{\omega}^{<0}$ is the same as the set of homology classes with negative self-intersection and having a simple rational pseudo-holomorphic curve representative.
By  Proposition 4.2 in \cite{Zha17}, we have

\begin{lma}\label{rat1}
Condition \ref{negsphemb} holds true for a symplectic rational surface $(X, \w)$ with $\chi(X) \leq 11$.
\end{lma}

Condition \ref{negsphemb} has crucial consequences for cusp curve decompositions and the prime sets.


\begin{lma} \label{general cusp}
Assume Condition \ref{negsphemb} holds.
  Let $\mC$ be an admissible set and  $A\in \mathcal S_{\w}^{}$.  For  $J\in \mJ_{\mC}$,
   suppose  $A$ is represented by a stable rational curve.
  Then  the cusp decomposition of the class $A$ for $J\in \mJ_{\mC}$ is of the form:
  \begin{equation} \label{cusp curve}A=\sum_{\alpha}  r_{\alpha}[C_{\alpha}] + \sum_{\beta}p_{\beta} [C_{\beta}]+
  \sum_{\gamma}q_{\gamma} [C_{\gamma}],
  \end{equation}
   where $r_{\alpha}, p_{\beta}, q_{\gamma}$ are positive integers, $C_{\alpha}, C_{\beta}, C_{\gamma}$ are simple $J$-holomorphic rational curves with
 $[C_{\alpha} ]^2\leq-2, [C_{\beta}]^2=-1,  [C_{\gamma}]^2\geq 0$. Moreover,

\begin{itemize}
\item $C_\alpha$ is embedded and hence $[C_{\alpha}]\in \mC$.

\item $C_{\beta}$ is embedded and hence $[C_\beta]\in \mS_\w^{-1}$.

\item $A\cdot [C_{\gamma}]\geq 0$.

\item $A\ne [C_\alpha], A\ne [C_\beta], A\ne [C_\gamma]$ for any $\alpha, \beta, \gamma$.
\end{itemize}

\end{lma}

\begin{proof}
 The first and second bullets follow from Condition \ref{negsphemb}.

  Notice that,  each $[C_{\gamma}]$ has $[C_{\gamma}]^2\geq 0$, by the positivity of intersection,  $[C_{\gamma}]$ pairs positively with any curve class in the decomposition  of $A$. In turn, we have
 $A\cdot \sum_{\gamma}q_{\gamma} [C_{\gamma}] \geq 0$.

The last bullet follows from the positivity of area.
   \end{proof}

Now we examine the case  $A\in \mathcal S_{\w}^{-1}$.

\begin{dfn}\label{enhancement}
Given an admissible set $\mC \subset S_{\w}^{\leq-2}$, let
$S_{\w}^{-1}(\mC)=\{A\in S_{\w}^{-1}|A\cdot A_i\geq 0 \hbox{ for any $A_i\in \mC$}\}$.
\end{dfn}

\begin{dfn}\label{jd}
For a subset $\mfD \subset \mS_{\w}^{-1}$, let $\mJ^{\mfD}$ denote the set of $J\in \mJ_{\w}$ such that   there is  an embedded $J$-holomorphic rational curve
in  each class in $\mfD$.

\end{dfn}

Note that here $\mfD$ is different from the admissible label set $\mC$ in Definition \ref{fine decomposition} since  $\mC\subset \mS_{\w}^{\leq -2}.$
We will next relate the prime set $\mJ_{\emptyset}=\mJ_{open}$   with $\mJ^{\mfD}$ under Condition \ref{negsphemb}.

\begin{lma}\label{open and -1}  Assume Condition \ref{negsphemb} holds.
  Let $\mC$ be an admissible set and  $A\in \mathcal S_{\w}^{-1}$.
  If  $A$ is represented by a stable rational curve then
Then  the cusp decomposition \eqref{cusp curve} of the class $A$ for $J\in \mJ_{\mC}$ has the additional property:
$A\cdot[C_\beta]\geq 0$. Consequently, $A$ is represented by  an embedded $J$-holomorphic rational curve if $A\in S_{\w}^{-1}(\mC)$.

If    $J\in \mJ_{open}$, then

  \begin{itemize}

\item   There are no simple $J$-holomorphic rational curves with self-intersection less than $-1$.

\item  For any  $A\in \mS_{\w}^{-1}$, there  is an   embedded  $J$-holomorphic rational curve in $A$ (in fact, unique).

\item   $\mJ_{open}\subset \mJ^{\mfD}$ for any $\mfD \subset \mS_{\w}^{-1}$.

\item   $\mJ^{\mfD}=\mJ_{open}$ if every class in $\mS_{\w}^{\leq -2}$ pairs negatively with some element in $\mfD$.

  \end{itemize}

\end{lma}

\begin{proof} Suppose  $A\in \mathcal S_{\w}^{-1}$ and $A$ is represented by a stable rational curve.
By Lemma \ref{general cusp}, the cusp decomposition \eqref{cusp curve} of $A$ satisfies $A\cdot [C_\gamma]\geq 0$.
  Recall that  $A \in S_{\w}^{-1}$,  and  notice that  $A\ne [C_{\beta}]$ for any $\beta$ since otherwise  $A$ is represented by $C_{\beta}$.
  By Proposition \ref{existence of curves}, $A\cdot \sum_{\beta}p_{\beta} [C_{\beta}] \geq 0$.
If $A\in S_{\w}^{-1}(\mC)$ then $A\cdot[C_\alpha]\geq 0$.
Combining the three inequalities, we have $A\cdot A\geq 0$. But this contradicts to  $A\cdot A=-1<0$.
Therefore $A$  has an embedded  $J$-holomorphic curve representative for each $J\in \mJ_{\mC}$ since, by Proposition \ref{existence of curves},
    $A$ is either represented by  an embedded $J$-holomorphic rational curve or
   represented by a $J-$holomorphic stable rational curve.

Now suppose $J\in \mJ_{\emptyset}=\mJ_{open}$ and we verify the four statements.

The first statement follows from the definition (\ref{fine decomposition}) of $\mJ_{open}$ and Condition \ref{negsphemb}.

The second statement follows from Proposition \ref{existence of curves} and the first statement. Notice that $\mS_\w^{-1}(\emptyset)=\mS_\w^{-1}$.

The third statement that $\mJ_{open}\subset \mJ^{\mfD}$ for any $\mfD \subset \mS_{\w}^{-1}$ follows directly from the second statement.

For the last statement, it suffices to show that $\mJ^{\mfD}\subset \mJ_{open}$ by the third statement. If $J\in \mJ^{\mfD}$, then there is an embedded $J$-holomorphic rational curve
in each class in $\mfD$. The desired conclusion follows from the positivity of intersection.
\end{proof}

We will further study when $\mJ_{open}= \mJ^{\mfD}$   in  Section \ref{Dopen}.

\begin{rmk}
Two related subsets $\mJ_{top}$ and $\mJ_{good}$ were introduced in \cite{LZ15}.
A tamed $J$ on a rational manifold is called good (\cite{LZ15}, page 4; \cite{FM88} when $J$ is integrable) if (i)
there is a smooth genus one subvariety in the anti-canonical class $-K_J$, and
(ii) any irreducible genus zero subvariety of negative self-intersection is a $(-1)$-curve.
$J\in \mJ_{top}$ (\cite{LZ15}, Def 2.13) if  any  simple $J$-holomorphic curve with negative self-intersection is an embedded rational curve with self-intersection $-1$.
 Clearly, $\mJ_{top}\subset \mJ_{open}$ since the definition of $\mJ_{open}$ only exclude simple $J$-holomorphic {\bf rational} curve with  self-intersection at most $-2$.
 However, when $\chi(X)\leq 11$, $\mJ_{top} = \mJ_{open}$ by Proposition 4.2 in \cite{Zha17}.

 The claim in Lemma \ref {open and -1} that  if $A\in S_{\w}^{-1}(\mC)$  is represented by  an embedded $J$-holomorphic rational curve
 could be compared with \cite{MO15} Theorem 1.2.7 (iii).  Our proof is much easier since we assume Condition 1 holds.
\end{rmk}


\subsubsection {Condition \ref{negsphemb} and prime submanifolds}
  Recall   that $\mJ_{\w}$  is  metrizable (cf. \cite{FS90}).  Hence the prime subsets are also metrizable.
Here is the main result of this subsection.

  \begin{prp}\label{submfld}
  Assume   Condition \ref{negsphemb} holds.
  For an   admissible set $\mC$ with nonempty $\mJ_{\mC}$,
the prime set $\mJ_{\mC}$  is a paracompact Hausdorff submanifold with  $\cod(\mJ_{\mC})=\cod_{\mC} =\sum_{C_i\in \mC } \cod_{C_i}$.

In particular,   for each $A\in \mS_{\w}^{-k}$, $\mJ_{A}$   is  a paracompact Hausdorff submanifold with codimension $2k-2$.
 \end{prp}

To prove this result we prepare a couple of lemmas.

 \begin{lma}\label{l:dec}
Assume Condition \ref{negsphemb}. Suppose $\mJ_{\mC}$ and $\mJ_{\mC'}$ are two distinct prime subsets with
$J_0\in \overline{\mJ_{\mC}} \cap \mJ_{\mC'} \ne \emptyset$.
Then
there is a non-empty subset $\mC_{deg}=\{A^i\} \subset \mC$
such that $\mC \setminus \mC_{deg} \subset \mC'$, and
 each class $A^i$ is represented by a $J_0-holomorphic $ stable rational curve.

\end{lma}

 \begin{proof}
Since $J_0\in \overline{\mJ_{\mC}} \cap \mJ_{\mC'} \ne \emptyset$,  then   there is a convergent sequence  $\{J_n\} \subset \mJ_{\mC} $
such that $\{J_n\}\rightarrow J_0\in \mJ_{\mC'}.$
For $J_0$,  take all the elements in  $\mC$ that are not irreducibly   $J_0$-holomorphic, and
denote the subset by  $\mC_{deg}=\{A^i\}$.   $\mC_{deg}$ is non-empty since $\mC\ne \mC'.$ Note that any class in $\mC \setminus \mC_{deg}$ has an irreducible $J_0$-holomorphic curve  representative, and hence  $\mC \setminus \mC_{deg}\subset \mC'$.
By Theorem \ref{Gcpt}, for each  $A^i\in \mC_{deg}$, there is a $J_0$-holomorphic stable rational curve in the class $A$.


 \end{proof}

 \begin{lma}\label{tau}
 Assume Condition \ref{negsphemb}.
If $\mC' \subset\mC$ but $\mC'\ne \mC$, then $\overline \mJ_{\mC} \cap \mJ_{\mC'}= \emptyset.$
 \end{lma}
 \begin{proof}
  We argue by contradiction.  Suppose there exists some $J' \in \overline \mJ_{\mC} \cap \mJ_{\mC'}.$
 It follows from  Lemma \ref{l:dec} and \eqref{cusp curve} that,
  for some   $A\in  \mC\setminus \mC' $,  there is a cusp decomposition
   of the form:
$$   A=\sum_{\alpha}  r_{\alpha}[C_{\alpha}] + \sum_{\beta}p_{\beta} [C_{\beta}]+
  \sum_{\gamma}q_{\gamma} [C_{\gamma}],
$$
   where $r_{\alpha}, p_{\beta}, q_{\gamma}$ are positive integers, $C_{\alpha}, C_{\beta}$ are simple $J'$-holomorphic rational curves with
 $[C_{\alpha} ]^2\leq-2, [C_{\beta}]^2=-1,  [C_{\gamma}]^2\geq 0$. Moreover,

\begin{itemize}
\item $C_\alpha$ is embedded and hence $[C_{\alpha}]\in \mC$.

\item $C_{\beta}$ is embedded and hence $[C_\beta]\in \mS_\w^{-1}$.

\item $A\cdot [C_{\gamma}]\geq 0$.

\item   $\w(A)  >  \w([C_\alpha]),  \w([C_\beta]), \w( [C_\gamma])$ for any $\alpha, \beta, \gamma$.
\end{itemize}

 Since $J'\in  \mJ_{\mC'}$ we have $[C_{\alpha} ]\in  \mC'$, and hence $[C_{\alpha} ]\in  \mC$ by our assumption $\mC' \subset\mC$. Since $A\in   \mC$ and $[C_{\alpha} ]\in \mC$, we have   $A \cdot [C_{\alpha} ] \geq 0$ for each $\alpha$.  Notice that $A\cdot [C_{\gamma}]\geq 0$ for each $\gamma$.

 Since $A\cdot A<0$, there must exist $\beta_0$ such that $C_{\beta_0}\cdot A<0$.  Notice that $\omega(A)> \omega(C_{\beta_0})$.
 Fix $J\in \mC$.
 Since  $[C_{\beta_0}]\in \mS_\w^{-1}$, it must have $J-$holomorphic curve or stable curve representative. Since $A\in \mC$, an embedded representative of $[C_{\beta_0}]$  implies
 $A\cdot [C_{\beta_0}]\geq 0$.

 Thus we are left with the case of a $J-$holomorphic stable rational curve in the class $[C_{\beta_0}]$.
 However, since  $A\in \mC$ and $A\cdot [C_{\beta_0}]< 0$,   in such a $J-$holomorphic stable rational curve of $[C_{\beta_0}]$ there has to be an irreducible component in the class $A$
 because otherwise we would have $A\cdot [C_{\beta_0}]\geq 0$. But this implies that $\omega([C_{\beta_0}]>\w(A)$, contradicting to the inequality $\omega(A)> \omega(C_{\beta_0})$.


 \end{proof}

Finally, we verify the prime subsets are submanifolds.

  \begin{proof}
  [Proof of Proposition \ref{submfld}]
 First, note that $\mJ_{\mC}$ is a subset of
   $U_{\mC}$, which is a submanifold of $\mJ_{\omega}$ with
   codimension $\sum_{i\in I } cod_{C_i}$ by Proposition \ref{stratum}.
 Then we examine  $U_{\mC} \setminus \mJ_{\mC}$.
  $U_{\mC}$ is a
  disjoint union of $\mJ_{\mC'}$  where each  $\mC'$ is admissible and
  contains $\mC$ as a proper subset.
  And the union of
   these $\mJ_{\mC' }$ is relatively closed in $U_{\mC}$ by
   Lemma \ref{tau}.
   Hence $\mJ_{\mC}$ is itself a submanifold of codimension
  $ \sum_{i\in I } cod_{C_i}$.  The paracompactenss and Hausdorff property come from the
  metrizable property (cf. \cite{FS90}).
  \end{proof}


To further analyze the decomposition, especially  $\mX_2$ and $ \mX_4$,  we will restrict to the case of a reduced symplectic form.

\subsection{The   cone \texorpdfstring{ $P(X)$,}{Lg} the  root systems \texorpdfstring{ $R(X)$}{Lg}  and \texorpdfstring{ $\Gamma_L(X, \w)$}{Lg} }

\subsubsection{Reduced symplectic forms}
We recall the notion of reduced symplectic forms.

\begin{dfn}\label{reduced}
Let $X$ be $ {\CC P^2}\# k\overline{\CC P^2}$ with a standard basis $\{H, E_1, E_2, \cdots, E_{k}\}$  of $H_2(X;\ZZ)$.
A  class   $\nu H-\sum c_i E_i$  is called {\bf reduced} (with respect to the basis) if
$$
c_1\geq c_2 \geq \cdots \geq c_k>0
\quad \text{and} \quad \nu\geq c_1+c_2+c_3.
$$

Reduced cohomology classes are defined as the Poincar\'e dual of reduced homology classes. A symplectic form $\omega$ on $X$ is called reduced if $[\omega]$ is reduced.
A reduced symplectic class is the class of a reduced symplectic form.
\end{dfn}

To us, the importance of the notion is the following result [\cite{GaoHZ}, \cite{LL01}, \cite{KK17}, and its Math Review]:
\begin{thm} \label{redtran}
For a rational surface $X= {\CC P^2}\# k\overline{\CC P^2}$,  every class with positive square in $H^2(X;\RR)$ is equivalent to a unique reduced class under the action of $\rm{Diff}^+(X)$.
If a symplectic form $\omega$ on $X$
is reduced, then its canonical class is  $$ K_{\w}=-3H +\sum^{k}_{i=1} E_i.$$
When  $3\leq k\leq 8$, any reduced class is represented by a symplectic form.
When $k\leq 2$, any reduced class with $\nu>c_1+c_2$ is represented by a symplectic form.



  \end{thm}

For the history of the first statement, see the MathSciNet Review of \cite{KK17}, where a proof in this generality is also given following \cite{GaoHZ}.  The rest of the theorem follows from Theorem 3 and  Proposition 3.6 in \cite{LL01}.

\begin{lma}\label{minemb}

For any reduced symplectic form $\w$ on  ${\CC P^2}\# k\overline{\CC P^2}, k\geq 3$,   $E_k$
has the smallest $\w$-area in the set $\mathcal S_{\w}^{-1}$.
\end{lma}

\begin{proof}
 Any class in $\mathcal S_{\w}^{-1}$ is of the form $E_i, 1\leq i \leq k$, or  $A=dH-\sum_{i=1}^k a_i E_i$ is in $\mathcal S_{\w}^{-1}$ for $d>0, a_i\geq 0$ (see eg \cite{LW12}).
 By the reduced assumption,  $\omega (E_1) \geq \cdots \geq \omega (E_i)  \geq \cdots \geq \omega (E_k)$ and $\w(H)\geq  \omega (E_i)+\omega (E_j) +\omega (E_k),  \forall$ distinct $i,j,k.$

Suppose $A=dH-\sum_{i=1}^k a_i E_i$ is in $\mathcal S_{\w}^{-1}$.   By the adjunction formula, $K_{\omega}\cdot A=-1$ and hence $3d=1+\sum_{i=1}^n a_i $.    Also by positive paring with the classes $H-E_1$ and $H-E_2$, $a_1\leq d $ and $a_2\leq d $. In particular,  $a_1+a_2\leq 2d$.
Therefore it is easy to see that we can write $A$ as a sum
$A= U_1+\cdots+U_{d-1}+V$, where each $U_a$ is of the form $H-E_p-E_q-E_r$
with no repeated $E_1$ or $E_2$,  and  $V$ is of the form  $H-E_i-E_j$.
Observe that   $\omega(U_a)\geq0$    and $\omega(V)\geq \omega(E_k)$ by the reduced condition. Thus the lemma would follow from such a decomposition of $A$.

\end{proof}

\subsubsection{$P(X_k)$ as a polyhedral cone for $3\leq k \leq 8$}\label{3to8}

\begin{dfn}
Let $X_k= {\CC P^2}\# k\overline{\CC P^2}$.  Its normalized reduced symplectic cone $P_k=P(X_k)$  is defined as the space of reduced symplectic classes having area 1 on $H$.  We represent such a class
 by $(1|c_1, \cdots, c_k)$,     or $(c_1, \cdots, c_k)\in \RR^k$.
 \end{dfn}

 When $ k\leq8,$
 there is a distinguished  class in $P_k$,  $M_k= (1|\frac13, \cdots, \frac13)=(\frac13, \cdots, \frac13),$  which is called the (normalized) monotone class.
 When $3\leq k\leq8,$ we will show that $P_k$ is a  polyhedral cone
 that are described by $M_k$ and
the following  $k$ classes of square $-2$:

\begin{equation}
  l_1 = H-E_1-E_2-E_3,  \quad l_2 = E_1-E_2,\quad  \cdots \quad, \quad l_k =  E_{k-1}-E_k.  \label{simroot}
\end{equation}

\begin{prp}\label{nrsc}
For $X_k=\CC P^2 \# k{\overline {\CC P^2}}, 3\leq k\leq8,$  the  normalized reduced symplectic cone $P_k$  is a polyhedral cone
in $\mathbb R^k$ with the  vertex at $M_k$ and the convex base
in the $c_1c_2\cdots c_{k-1}$ (i.e. $c_k=0$) hyperplane  generated by the following  $k$ points $G_i$:
$$G_1=(0,\cdots,0),
 G_2=(1, 0,\cdots, 0),
G_3=(\frac{1}{2}, \frac{1}{2}, 0, \cdots,0),$$
$$G_4=( \frac{1}{3}, \frac{1}{3}, \frac{1}{3}, 0,\cdots,0),
  \cdots ,
G_k=(\frac{1}{3},\cdots, \frac{1}{3}, 0).
$$
The symplectic classes on each edge $M_kG_i$ are characterized by the property of pairing trivially with  $l_j$ for any $j\ne i$ and positively on $l_i$.

 Consequently, the reduced symplectic classes are characterized as the symplectic classes which are positive on each $E_i$ and non-negative on each $l_i$.
\end{prp}

\begin{proof}
 When $3\leq k\leq8,$  any reduced class of $X_k$ is a reduced symplectic class.
And a  normalized reduced class satisfies
 \begin{equation}\label{normalized reducedHE}
 1\geq c_1+c_2+c_3, \quad  c_1\geq c_2, \quad  c_2\geq c_3, \quad   \cdots, \quad  c_{k-1}\geq c_k, \quad   c_k>0.
\end{equation}

  Let $\Psi$ be the  translation moving $M_k=(\frac13, \cdots, \frac13)$ to $0.$ Under this linear translation,  $(1|c_1, \cdots, c_k)$ is moved to  $x=(x_1, \cdots, x_k)=(c_1-\frac13, \cdots, c_k-\frac13)$, and the normalized reduced condition (\ref{normalized reducedHE}) can be written as the  $k$ homogeneous conditions:
  $$  0\geq x_1+x_2+x_3, \quad  x_1-x_2\geq 0, \quad x_2-x_3\geq 0,  \quad \cdots, \quad    x_{k-1}- x_k\geq 0, \quad x_k>-\frac{1}{3}.$$
Clearly, $\Psi(P_k)$ has only one vertex at the origin and its opposite face is open and at the hyperplane $x_k=-\frac{1}{3}$.
There are $k$ inequalities of the form $\geq$ in (\ref{normalized reducedHE}). Setting $c_k=0$ and all  of the $k$ inequality $\geq $ to be equality except the $i$-th one, we obtain the $k$ points $G_i$  in the
$c_1\cdots c_{k-1}$ hyperplane.
The rays $M_kG_i$ are clearly extremal rays. Notice that $M_k$ pairs trivially with  each $l_j$,  and $G_i$ pairs trivially with each $l_j$ for each $j\ne i$. It follows that $M_kG_i$ pairs trivially with each $l_j$ except for $j=i$.
\end{proof}

We remark that the  $K-$symplectic  cone in \cite{LL01}  for $K=-3H+\sum_{i=1}^k E_i$ has an explicit geometric description
in Proposition 3.2  \cite{Zha17}.   This  $K-$symplectic  cone  is acted on by  the Cremona group, where the Cremona group is the subgroup of $Aut^+(H^2(X, \mathbb Z))$ preserving $K$.  The cone of reduced symplectic classes is just the fundamental domain  of  the $K-$symplectic cone under this action, and
the normalized reduced cone $P_k$  in Proposition \ref{nrsc} is the   slice in the fundamental domain with the $H$ coefficient being $1$.

\begin{dfn}  A $p-$dimensional {\bf open face} of $P_k$ is defined as the {\bf interior} of the convex hull of $M_k$ together with $p\leq k$ points in the set $\{G_i\}.$  $P_k$ has
$2^k$ open faces in total:
a unique zero dimensional open face $M_k$; $k$ one dimensional open faces,
and generally, $\binom{k}{p}$ open faces of dimension $p$.

Our convention is to denote an open face with  vertices $v_1, v_2, \cdots, v_l$ simply by $v_1v_2\cdots v_l$.
\end{dfn}

The top dimensional open face is the space of reduced classes given by the $k$  strict inequalities:   $\lambda:=c_1+c_2+c_3<1; c_1>\cdots >c_k.$
An open face of  codimension $l$ is when
 $l$  of those $``>"$ is turned  into $``="$.

Also note that, when projeced  onto the hyperplane $c_k =0,$ $M_k$ is sent to $M_{k-1}$ and $P_k$ is sent to $P_{k-1}$.

 In picture \ref{3c}, the tetrahedron $MOAB$ describes $P_3$, where $M=M_3, O=G_1, A=G_2, B=G_3$. There are 8 open faces:
 $$M, MO, MA, MB, MOA, MOB, MAB, MOAB.$$
 Notice that the triangle $OAB$ is completely outside $P_3$.
 We also remark that the change from $G_i$ to $MOAB$ notation will be used frequently.

\begin{figure}[ht]
\begin{center}
\begin{tikzpicture}
[scale=4,
axis/.style={->,blue,thick},
vector/.style={-stealth,red,very thick},
]
\coordinate (O) at (0,0,0);
(-0.1,0,0)node{$O$};
\draw[axis] (0,0,0) -- (1.2,0,0) node[anchor=north east]{$c_2$};
\draw[axis] (0,0,0) -- (0,1.3,0) node[anchor=north west]{$c_3$};
\draw[axis] (0,0,0) -- (0,0,1.4) node[anchor=south]{$c_1$};
\draw[vector]  (0.38,0.3,0.38)--(0,0,0);
\draw[dashed,red](0,0,0) -- (0.5,0,0.5);
\node[draw] at (0.6,0,0.6){$B:(\frac12, \frac12, 0)$};
\draw[dashed,white] (0,0,0)--(0,0,0.8);
\node[draw] at (0,-0.1,0.7){$A:(1,0,0)$};
\draw[vector]  (0.38,0.3,0.38)--(0,0,0.8) ;
\draw[dashed,red] (0.5,0,0.5) -- (0,0,0.8);
\draw[vector] (0.38,0.3,0.38)--(0.5,0,0.5);
\node[draw] at (0.4,0.3,0){$M=M_3:(\frac13,\frac13,\frac13)$};
\draw[fill=red] (0.38, 0.3, 0.38) circle (.03);
\end{tikzpicture}
\caption{Normalized Reduced symplectic cone of $\CC P^2 \# 3\overline{\CC P^2 }$}\label{3c}
\end{center}
\end{figure}

\subsubsection{The root system $R_k$ and edges of $P_k$}

Recall that a  root system $\Phi$ in an inner product  vector space $(E, <, >) $  is a finite spanning set of non-zero vectors, called roots, that satisfies the following 2 conditions:

a). If  $\alpha\in \Phi$ then $n\alpha \in \Phi$ if and only if $n=\pm 1$.

b). For any two roots $\alpha, \beta \in \Phi$,  the number $2\frac{< \alpha, \beta>}{<\alpha, \alpha>}$ is an integer and $\sigma_{\alpha}(\beta)=\beta-2\frac{< \alpha, \beta>}{<\alpha, \alpha>}\alpha \in \Phi$.

Recall that, given a root system $\Phi$  in $(E, <, >)$,  a set $\triangle$ of simple roots is a subset of $\Phi$ which forms a basis of $E$ and has the additional property that
every root in $\Phi$ is an integral linear combination of elements of $\triangle$ with the coefficients either all non-negative or all non-positive.
 Given a set $\triangle$ of simple roots, the set of  roots which are non-negative (non-positive)  integral linear combinations of elements of $\triangle$ is called the set of positive (negative)  roots, and it is denoted by $\Phi^+$ ($\Phi^-$). Clearly, $\Phi=\Phi^+\sqcup \Phi^-$ and $|\Phi^+|=|\Phi^-|=\frac{1}{2}|\Phi|$.

We  slightly reformulate a beautiful fact in \cite{Man86} (Theorem 23.9 in pages 115-116,  see also page 1 in \cite{LZ14}).  For    $X_k$  with $3 \leq k\leq 8$, define the set
\begin{equation}\label{smooth root system}
R(X_k)=R_k :=  \{ A \in H_2(X_k,\ZZ)   \mid \left<A,K_k\right> = 0, \quad \left<A,A\right> = -2 \},
\end{equation}
where $K_k=-(3H-E_1-\cdots-E_k)$.
It is straightforward to check that the root system $R_k$ satisfies both conditions a) and b) and hence forms a root system.
$R_k$ as a root system is described  in the  table below,

\[ \begin{array}{c|cccccc}
k & 3 & 4 & 5 & 6 & 7 & 8 \\
\hline
R(X_k)&\aA_1\times \aA_2 &\aA_4 & \DD_5 & \EE_6 & \EE_7 & \EE_8\\
|R(X_k)| & 8 & 20 & 40 & 72 & 126 & 240 \\
\end{array} \]

Moreover,
\begin{itemize}
\item The classes $l_i$  in \eqref{simroot} are in the root system $R_k$ and provide a  canonical choice of  simple roots of $R_k$, which form the vertices of the Dynkin diagram.

\item
As remarked after Proposition \ref{nrsc}, the reduced cone is the fundamental domain of some Weyl group action, and  positive roots pairs non-negative with any vector in the Weyl chamber, which means  they pair non-negatively with an arbitrary reduced form.
\item By  Proposition \ref{nrsc}   these simple roots $l_i$  correspond to the edges $M_kG_i$ of $P_k$ in the sense that  the symplectic classes on each edge $M_kG_i$ are characterized by the property of pairing trivially with  $l_j$ for any $j\ne i$ and positively on $l_i$.
When $k=3$, with the notation $MOAB$, $l_1$ corresponds to $MO$, $l_2$ corresponds to $MA$ and $l_3$ corresponds to $MB$.  We denote this correspondence $l_i \leftrightarrow MG_i.$

\end{itemize}

\subsubsection{Lagrangian root systems for $3\leq k\leq 8$} \label{Lagroot}

The class of a  Lagrangian sphere in a symplectic 4-manifold $(X, \omega)$ has square $-2$ and pairs trivially with $[\omega]$ and $K_{\omega}$.
A class is called a Lagrangian sphere class if it is represented by an embedded Lagrangian sphere. Let $\Gamma_L(X, \omega)$ be the set of Lagrangian sphere classes.
The following observation should be known to experts.

\begin{lma} \label{roots}
Suppose $(X, \w)$ is a symplectic 4-manifold with a finite $\Gamma_L(X, \omega)$. Then   $\Gamma_L(X, \omega)$  is  a  root system in the span of the root system $\Gamma_L(X, \omega)$.
\end{lma}

\begin{proof}

Condition a) is clearly satisfied for $\Gamma_L(X, \omega)$ since any class $\alpha \in \Gamma_L(X, \omega)$ has $\alpha\cdot \alpha =-2$ and $n\alpha\cdot n\alpha =-2 n^2=-2$ only if $n=\pm 1$.
As for condition b), it is satisfied by the construction of Seidel's Lagrangian Dehn twist $S_L$ for a Lagrangian sphere $L$. $S_L$ is a symplectomorphism acting on $H_2(X; \RR)$ as a reflection through the hyperplane perpendicular to $[L]$. In particular, if $\alpha$ is represented by a Lagrangian sphere $L$  and $\beta$ is represented by a Lagrangian sphere $L'$, then
$S_L(L')$ is a Lagrangian sphere in the class $\sigma_{\alpha}(\beta)$.
\end{proof}

\begin{prp} \label{MLS}

Suppose $\omega_{mon}$ is a monotone symplectic form on $X_k$ with $3\leq k\leq 8$, then $\Gamma_L(X_k, \omega_{mon})=R_k$ as root systems.

For a reduced symplectic form $\omega$ on $X_k$,
  $\Gamma_L(X_k, \omega)$ is a sub-root system of  $\Gamma_L(X_k, \omega_{mon})$, and  a canonical choice of simple roots  consists of those $l_i\in \Gamma_L(X_k, \omega_{mon})$ which pair trivially with $[\omega]$.  $[\omega]$ pairs non-negatively with the positive roots of  the root system $\Gamma_L(X_k, \omega_{mon})$.

\end{prp}

\begin{proof}

  Clearly, the root system $R_k$ contains all the Lagrangian sphere classes  of $(X_k, \omega_{mon})$ since $K_{\w_{mon}}=K_k$.
On the other hand, by \cite{LW12},   a class $A$ of a rational surface with $A\cdot A=-2$ is a Lagrangian sphere class if $A$ is represented by a smoothly embedded sphere and $A$  pairs trivially with $[\omega]$ and $K_{\omega}$.
 Since $\chi(X_k) \leq 12,$ by \cite{LL06}, every square $(-2)$-class of $X_k$ can be represented by a smoothly embedded sphere. Since $[\omega_{mon}]$ is proportional to $K_{\w_{mon}}$, any class in the root system $R_k$ satisfies the criterion in \cite{LW12}.

Now suppose that $\w$ is a reduced symplectic form.
Since $K_{\w}=K_{\w_{mon}}$ for a reduced symplectic form $\w$,  $\Gamma_L(X_k,\omega)$ is contained in  $\Gamma_L(X_k, \omega_{mon})$, and hence by Lemma \ref{roots}, $\Gamma_L(X_k, \omega)$ is a sub-root system of $\Gamma_L(X_k, \omega_{mon})$.
Explicitly, suppose $[\w]$ lies in  a $p-$dimensional open face, which we denote by $F(\omega)$. If  $M_k G_{i_1},$ $ M_k G_{i_2},\cdots M_k G_{i_p}$  are the  edges of $F(\omega)$, then $[\omega]$ is positive on the corresponding simple roots $l_{i_j}$ of
$\Gamma_L(X_k,\w_{mon})$
and vanish on the remaining simple roots. Therefore $[\omega]$ vanishes exactly on the   positive roots of  $\Gamma_L(X_k,\w_{mon})$ generated from these simple roots $l_{i_j}$
and is positive on the remaining positive roots of  $\Gamma_L(X_k,\w_{mon})$. In other words,
$\Gamma_L(X_k,\omega)$  contain exactly  the   positive roots of  $\Gamma_L(X_k,\w_{mon})$ generated from these simple roots $l_{i_j}$.
\end{proof}

Let  $N_{\w}$ be the number of $\omega$-symplectic $(-2)$-sphere classes. Note that $N_{\w}$ and $\Gamma_L(X_k,\w)$ are both invariant in any given open face.  Let   $N_{\w,L}$ be the number of  $\omega-$Lagrangian sphere classes up to a change of the sign.
    Following from the above discussion and denoting the set of positive roots  of $\Gamma_L(X_k,\w_{mon})$ by $R^+(X_k)$, we have
\begin{equation}\label{SL}
    N_{\w}+N_{\w,L}= |R^+(X_k)|=\frac{1}{2} |R_k|.
\end{equation}

We take the 3-point blow up as an example to illustrate the above correspondence between the simple roots and the edges of the cone. In this case, the simple roots are $  l_1 = H-E_1-E_2-E_3,  \quad l_2 = E_1-E_2,  l_3= E_2-E_3; $ and
$$R^+(X_3)=\{ H-E_1-E_2-E_3,  E_1-E_2, E_2-E_3, E_1-E_3\}, \quad N_{\w}+N_{\w,L}=4.$$

We describe the  Lagrangian root system on each open face of $P_3$, which is the tetrahedron $MOAB$ in  Figure \ref{3c}.

\begin{itemize}
\item  The monotone vertex  $M$.  $N_{\w}=0$ in this case.
 And 3 Lagrangian simple roots $MO\leftrightarrow l_1= H-E_1-E_2-E_3,$ $MA\leftrightarrow l_2= E_1-E_2,$ $MB\leftrightarrow l_3= E_2-E_3$ form $R(X_3)$ with the Dynkin diagram:

  \begin{center}

\begin{tikzpicture}

    \draw (1,0) -- (2,0);

    \draw[fill=red] (1,0) circle (.1);
    \draw[fill=red] (2,0) circle (.1);
    \draw[fill=red] (3,0) circle (.1);

    \node at (1 ,0.4) {$MA$};
  \node at (2,0.4) {$MB$};
   \node at (3,0.4) {$MO$};
    \node at (0,0) {$\aA_1\times \aA_2$};

\end{tikzpicture}

\end{center}

 \item The edge  $MO$.
 $\Gamma_L=\aA_2$,  obtained from $R(X_3)$ by removing $MO$:

  \begin{center}

\begin{tikzpicture}

    \draw (1,0) -- (2,0);

    \draw[fill=red] (1,0) circle (.1);
   \draw[fill=red] (2,0) circle (.1);
   \draw[dotted] (3,0) circle (.1);
   \node at (3,0.4) {$MO$};

   \node at (1 ,0.4) {$MA$};
   \node at (2,0.4) {$MB$};
   \node at (0,0) {$ \aA_2$};

\end{tikzpicture}

\end{center}

\item  The edge $MA$.  $\Gamma_L=\aA_1\times \aA_1$,  obtained from $R(X_3)$ by removing $MA$:
  \begin{center}

\begin{tikzpicture}

    \draw[dotted,thick]  (1,0) -- (2,0);

    \draw[dotted](1,0) circle (.1);
    \draw[fill=red] (2,0) circle (.1);
    \draw[fill=red] (3,0) circle (.1);

  \node at (1 ,0.4) {$MA$};
  \node at (2,0.4) {$MB$};

   \node at (3,0.4) {$MO$};
    \node at (0,0) {$\aA_1\times \aA_1$};

\end{tikzpicture}

\end{center}
\item The edge $ MB$.  $\Gamma_L=\aA_1\times \aA_1$, obtained from $R(X_3)$ by removing $MB$:

  \begin{center}

\begin{tikzpicture}

    \draw [dotted,thick] (1,0) -- (2,0);

    \draw[fill=red] (1,0) circle (.1);
    \draw[dotted] (2,0) circle (.1);
    \draw[fill=red] (3,0) circle (.1);

    \node at (1 ,0.4) {$MA$};
  \node at (2,0.4) {$MB$};

   \node at (3,0.4) {$MO$};
    \node at (0,0) {$\aA_1\times \aA_1$};

\end{tikzpicture}

\end{center}

\item Three  open faces of dimension 2:  $MOA,$ $MOB$ $MAB$.   $\Gamma_L=\aA_1$ lattice of node $MB$, $MA$ and $MO$ respectively(by removing all other vertices).

\item The top dimensional open face $MOAB$,  where all
$4$  spherical $(-2)$-class are  symplectic and the Lagrangian system  $\Gamma_L$ is $\emptyset$.
\end{itemize}

\subsubsection{$R(X)$, $P(X)$ and $\Gamma_L(X, \omega)$  when $\chi(X)\leq 5$}\label{chileq5} The main purpose of this subsection is to describe $R(X)$, $P(X)$ and $\Gamma_L(X, \omega)$ for rational surfaces with $\chi(X)\leq 5$ and
show that the equation \eqref{SL} continues to hold for these rational surfaces in  Lemma \ref{5sl}.
Such  rational surfaces  are $S^2 \times S^2$ and $X_k, k=0,1,2$.
We start with $X_k, k\leq 2$.

\medskip

\noindent $\bullet$ $X_2, X_1, X_0$.

For $k=0,1,2$, the root system $R(X_k)$ is similarly defined via \eqref{smooth root system} and $K_k=-3H+\sum_{i=1}^k E_i$,  and it is easy to see that
$$ R(X_0)=R(X_1)=\emptyset, \quad R(X_2)=\{\pm (E_1-E_2)\}=\aA_1,  $$
with $R^+(X_2)=\{E_1-E_2\}$.

For the normalized reduced symplectic cone, notice that $P_2$ can be obtained by  projecting $P_3$ (cf.  Figure \ref{3c}) onto the plane $c_3 =0$.

\begin{figure}[ht]
\begin{center}
\begin{tikzpicture}
[scale=4,
axis/.style={->,blue,thick},
vector/.style={-stealth,red,very thick},
]
\coordinate (O) at (0,0,0);(-0.1,0,0)node{$O$};
\draw[axis] (0,0,0) -- (1,0,0) node[anchor=north east]{$c_1$};
\draw[axis] (0,0,0) -- (0,1,0) node[anchor=north west]{$c_2$};
\draw[red, thick]  (0.4,0.4,0)--(0,0,0) node[anchor=north ]{$c_1=c_2$};
\draw[green,thick] (0.4,0.4,0)--(0.8,0,0);
\draw[white] (0,0,0)--(0.8,0,0);
\node[draw] at(0.8,-0.15,0){$A:(1,0)$};
\node[draw] at(0.4,0.5,0){$B:(\frac12,\frac12)$};
\node[draw] at(0,0.33,0){$M_2:(\frac13,\frac13)$};
\draw[fill=red] (0.27, 0.27, 0) circle[radius=.02];
\draw[fill=gray] (0.27, 0, 0) circle[radius=.02 ];
\draw[fill=green] (0.4, 0.4, 0) circle[radius=.02 ];
\draw[fill=white] (0.8, 0, 0) circle[radius=.02 ];
\node[draw] at(0.27,-0.15,0){$M_1:(\frac13,0)$};
\end{tikzpicture}\caption{Normalized reduced cone when $\chi(X)\leq 5$ }\label{2c}
\end{center}
\end{figure}

In fact, we can visualize $P_k, k=0, 1,2,$  via Figure \ref{2c}, which describes the closure of $P_2$. We  label the possible monotone points as $M_1, M_2$.
For $k=2$, the cone $P_2$  is   the  triangle $OAB$ with the closed edges $AB$ and $OA$ deleted, i.e. $\{ (c_1, c_2)|1>c_1+c_2> c_1\geq c_2>0\}$.
For $k=1$, the cone $P_1$  is  the projection of $P_2$ onto the $c_1$-line. Explicitly, $P_1$ is  the horizontal open interval $OA$, i.e.   $\{ c_1\in (0,1)\}$. For $k=0$, the cone $P_0$  is   the point $O$.

For $X_0$ or $X_1,$ there are no Lagrangian sphere classes and  the Lagrangian root system is  $\emptyset.$
For $X_2$,  the Lagrangian system is empty at the $2$-dimensional open face,  and $\aA_1$ at the  $1$-dimensional open face $OB$.

\medskip

\noindent $\bullet$ $S^2\times S^2$.

 Let $\{F_1, F_2\}$ be the natural basis of $H_2(S^2\times S^2)$, which is also a basis of $H^2(S^2\times S^2)$ via Poincar\'e duality.
Let  $K=-2F_1-2F_2$  and use it to define
the root system $R(S^2\times S^2)$ via \eqref{smooth root system}. Clearly, $R(S^2\times S^2)=\{\pm (F_1-F_2)\}$ is just $\aA_1$, with $R^+(S^2\times S^2)=\{F_1-F_2\}$.


 Up to the action of $\Diff^+(S^2\times S^2)$ and scaling, every symplectic class is equivalent to
a unique class of the form    $F_1+\mu F_2$  for some  $\mu\geq 1$.
 Hence, the coefficient $\mu$, or equivalently, the ratio,  identifies a fundamental domain with the half-open and half-closed interval $[1, \infty)$.
 We simply define the normalized reduced symplectic cone $P(S^2\times S^2)$ to be this $\mu$ interval.
 Via the ratio $c_1/c_2$, it actually appears in Figure \ref{2c} as the open interval $BA$ together with the endpoint $B$.
 The explanation is that blowing down an $H-E_1-E_2$ symplectic sphere in $(X_2, \w)$ gives rise to a symplectic $S^2\times S^2$, and this corresponds to
 projecting $P_2$ to the interval $[B,A)$.
 Notice that $B$, the image of $M_2$,  is the class of a monotone form on $S^2\times S^2$.

 As for the Lagrangian system,
there is  either zero or  one Lagrangian sphere class up to sign and  the Lagrangian root system is $\emptyset$ or   $\aA_1$ respectively. Precisely,   $\Gamma_L(S^2\times S^2, \w_{mon})=\aA_1$, and $\Gamma_L(S^2\times S^2, \w)=\emptyset$ for any non-monotone $\w$. In terms of  $P(S^2\times S^2)=[1, \infty)$, the Lagrangian system is
empty on the $1$-dimension open face $(1, \infty)$ and $\aA_1$ at
the $0$-dimensional open face $\{1\}$.
 In terms of  the $[B,A)$ interval representation of $P(S^2\times S^2)$ in Figure \ref{2c}, the Lagrangian system is
empty on open  $BA$ and $\aA_1$ at
 $B$.

Finally, notice that the relation \eqref{SL} still holds when $\chi(X)\leq 5$.   We state this fact as a lemma. 

\begin{lma}\label{5sl} For any symplectic rational surface $(X, \w)$ with  $\chi(X)\leq 11$, we have
\begin{equation}\label{SL'}
    N_{\w}+N_{\w,L}= |R^+(X)|=\frac{1}{2} |R(X)|.
\end{equation}
\end{lma}

This observation will be useful in proving Corollary \ref{4Q}.

\begin{rmk}\label{geq9}
 When $\chi(X)\geq 12$, the normalized reduced symplectic cone is no longer a polytope. It is cut by one more quadratic equation and a K-positive linear equation. Further, the Cremona transform group is infinite, and the set of Lagrangian sphere classes forms a (generalized) root system guided by affine Kac-Moody algebra (cf. \cite{GaoHZ}).
 \end{rmk}




\section{Level 2 stratification   of \texorpdfstring{$\mJ_{\omega}$ when $\chi\leq 8$} {Lg}}\label{s:ACS}

In this section, we prove Theorem \ref{level 2} and Corollary \ref{Jopen}  for symplectic rational 4-manifolds with Euler number no larger than 8.

 For the convenience of computation, throughout this section we identify $\CC P^2  \# k\overline {\CC P^2}, k\geq 2,$ with $S^2\times S^2 \# (k-1)\overline{\CC P^2}, k\geq 2, $ and
 use  two natural bases for
 $H_2$.
Let  $\{B, F, E'_1,\cdots, E'_{k-1}\}$  be the natural basis of $H_2(S^2\times S^2 \# (k-1)\overline{\CC P^2};\ZZ)$.
 Then the transition from the  $\{B, F, E_i'\}$ basis to the  $\{H, E_i\}$ basis  is  explicitly given by
 \begin{align}\label{BH}
 B=H-E_2,\nonumber \\F=H-E_1, \nonumber \\E'_1=H-E_1-E_2,\\E'_i=E_{i+1},\forall i\geq 2\nonumber,
 \end{align}
with the inverse transition given by:
  \begin{align}\label{HB}
 H=B+F-E'_1,\nonumber \\E_1=B-E'_1, \nonumber \\E_2=F-E'_1,\\E_j=E'_{j-1},\forall j>2\nonumber .
 \end{align}
Thus,  $\nu H-c_1E_1 - c_2E_2 -\cdots -c_kE_k$ corresponds to
\begin{equation}\label{ctoa}
(\nu-c_1)B+ (\nu-c_2)F-(\nu-c_1-c_2)E_1'-c_3E_2'-\cdots -c_kE_{k-1}'.
\end{equation}

Note that in terms of the  basis $\{B,F, E'_1,\cdots, E'_{k-1}\}$,
up to scaling, a reduced class is of the form $B+\mu F-\sum_{i=1}^{k-1} a_i E'_i$ with
\begin{equation}\label{reducedBF}
 \mu \geq 1 > a_1 \geq  a_2 \geq \cdots \geq a_{k-1}>0
\quad \text{and} \quad a_1+a_2\leq 1.
\end{equation}
If a symplectic form $\omega$ on $X$
 is reduced, then using the basis $\{B,F, E'_1,\cdots, E'_{k-1}\}$, its canonical class is $$ K_{\w}=-2B-2F +\sum^{k-1}_{i=1} E'_i. $$

\subsection{The \texorpdfstring{$\mB, \mF, \mE$}{Lg} classification of \texorpdfstring{$\mathcal S_{\w}^{<0}$}{Lg} for a reduced form}


Here is our setting in this subsection: Suppose  $X=S^2\times S^2 \# n\overline{\CC P^2}, n \leq4,$  and $\omega$   is a reduced symplectic form in the class $B+ \mu F -\sum_{i=1}^n  a_i E'_i$ satisfying  \eqref{reducedBF}. We first make the following elementary observation, which will be crucial in this section.

\begin{lma} \label{reduceform}Suppose  $X=S^2\times S^2 \# n\overline{\CC P^2}$  and $\omega$   is a reduced symplectic form in the class $B+  \mu F -\sum_{i=1}^n  a_i E'_i$.
Then
\begin{equation} \label{e: reduced} \sum_{k=1}^n(a_k)^2\leq1, \quad \hbox{ if } n\leq4.
\end{equation}
\end{lma}
\begin{proof}If we consider the extreme value of the function $\sum_{k=1}^n (a_k)^2$ under the constrain given
by the closure of the reduced condition \eqref{reducedBF}, $a_i \in[0,1], a_i+a_j\leq 1$,
then the extreme value can only appear at $(1,0,\cdots,0) \quad \text{or} \quad (\frac12,\frac12,\cdots,\frac12)$,
meaning that $\sum_{i=1}^n (a_i )^2\leq max(1,n/4)$. Given $n\leq4,$ we have $\sum_{k=1}^n(a_k)^2\leq1$.
\end{proof}

\subsubsection{Constraints on simple \texorpdfstring{ $J$-}{Lg}holomorphic  curves for a reduced form}
The following is a key lemma providing basic constraints for simple $J$-holomorphic curves.


 \begin{lma}\label{classes}
Under the above setting,
if   $A=pB+qF-\sum r_iE'_i \in H_2(X;\mathbb{Z})$
is represented by   a simple  $J$-holomorphic curve $C$ for some $\omega$-tamed $J$,
then $p \geq 0$.

And if $p=0$, then $q=0$ or $1$.  If $p=1$, then $r_i \in \{ 0,1 \}$. If $p>1$, then $q\geq 1$.
\end{lma}

\begin{proof}
We start by stating three  inequalities: the area inequality \eqref{Area}, the adjunction inequality \eqref{adj},  the $r_i$ integer inequality \eqref{r integer}.

The area of the curve $C$  is positive and hence
\begin{equation}\label{Area}
 \omega(A)=p\mu+q -\sum_{i=1}^n a_ir_i > 0.
\end{equation}

By Theorem \ref{adj ineq theorem}
 the virtual genus $g_v(C)$ of the simple $J$-holomorphic curve $C$ (not necessarily embedded), defined by $g_v(C)= ([C]\cdot [C] -c_1(X, J)([C]))/2+1$, is non-negative.
Since  $J$ is $\w$ tamed and $\omega$ is reduced, $-c_1(X, J)$   is  the canonical class $K_{\omega}=-2B-2F+E'_1+\cdots +E'_n$. Note also that $g_v(C)$ is defined in terms of the homology class $[C]=A$ so
we  switch notation from $g_v(C)$ to $g_{\w}(A)$. Now
  the adjunction
inequality for  the simple $J$-holomorphic curve $C$ is of the following form:
 \begin{equation} \label{adj}
 0\leq 2g_{\omega}(A)=A\cdot A+K_{\w}\cdot A+2= 2(p-1)(q-1)-\sum_{i=1}^n r_i(r_i-1).
 \end{equation}

The third inequality is an  estimate of  the sum   $ -\sum_{i=1}^n r_i(r_i-1)$.  Since each $r_i$ is an integer, it is easy to see that
\begin{equation} \label{r integer}
 -\sum_{i=1}^n r_i(r_i-1) \leq 0,
  \end{equation}
and  $ -\sum_{i=1}^n r_i(r_i-1)=0$ if and only if $r_i=0$ or $1$ for each $i$.

\medskip
Now let us   divide into five  cases:

(i) $p=1$, (ii) $p>1$,  (iii)   $p<0,  q\geq1$,  (iv) $p< 0, q\leq  0$,  (v)  $p=0$.

\noindent {\bf Case (i)}. $p=1.$ Then $ -\sum_{i=1}^n r_i(r_i-1)=  2g_{\w}(A)  \geq 0.$
It follows from  \eqref{r integer} that   $ r_i(r_i-1)$ has to be 0 and hence $r_i\in \{ 0,1 \}$.

\noindent {\bf Case (ii)}.  $p>1$ and $q\leq 0$. Then
 $  -\sum_{i=1}^n r_i(r_i-1) =     2g_{\w}(A)- 2(p-1)(q-1)  \geq  0 - 2(p-1)(q-1) > 0$.  This is impossible. Therefore  $q\geq 1$ if $p>1$.

\smallskip

\noindent {\bf Case  (iii)}.  $p<0$ and $q\geq1$.

We show that this case is impossible.
Because $p \leq -1$, the adjunction inequality  \eqref{adj} implies that
$$0\geq -2g_{\omega}(A) \geq 4(q-1) + \sum_{i=1}^n  r_i(r_i-1)\geq (q-1) + \sum_{i=1}^n  r_i(r_i-1).$$
Applying the area equation \eqref{Area}, we have
$$(q-1) + \sum_{i=1}^n  r_i(r_i-1) >  (\sum_{i=1}^n a_ir_i -\mu p -1)+ \sum_{i=1}^n r_i(r_i-1).$$
Since  $-\mu p-1 \geq0$,
 $$(\sum_{i=1}^n a_ir_i -\mu p -1)+ \sum_{i=1}^n r_i(r_i-1)   \geq (\sum_{i=1}^n a_ir_i)+ \sum_{i=1}^n r_i(r_i-1)
   = \sum_{i=1}^n r_i(r_i-1+ a_i).$$
For any integer $r_i$ we have $ r_i(r_i-1+ a_i) \geq 0$  due to the  reduced condition   $1-a_i\in(0,1)$.
Therefore we would have $-2g_{\omega}(A)>0$, which is a contradiction.

\noindent {\bf Case (iv)}.     $p< 0, q\leq  0$.

We show this case is also impossible.  This will follow from  the following estimate, under a   slightly general assumption:
 \begin{equation} \label{e: g, p, q}
0\leq 2g_{\omega}(A)\leq 1+ |p|+|q|-p^2-q^2, \quad\quad\quad\quad    \hbox{if }   p\leq 0, q\leq 0.
\end{equation}

{\bf Proof of the inequality  \eqref{e: g, p, q}.}
In order to estimate  $ -\sum_{i=1}^n r_i(r_i-1)$  we  rewrite  the sum
\begin{equation}\label {e: positive-negative}  \sum_{i=1}^n r_i=\sum_{k=1}^ur_k +\sum_{l=u+1}^n r_l,
\end{equation}
where
each $r_k$ is negative and  each $r_l$ is  non-negative.

Since $p\leq 0, q\leq 0$, the area inequality  \eqref{Area} takes the following form:
\begin{equation}  \label {e: a_i r_i} -\sum a_ir_i >(|p|+|q|).
\end{equation}
Note that  there exists at least one negative $r_i$ term, i.e. $u\geq 1$ in \eqref{e: positive-negative}.
 An important consequence is
 \begin{equation}\label{e:  crucial}
  \sum_{k=1}^u  a_k r_k \leq  \sum_{i=1}^n  a_ir_i<0, \quad  (\sum_{k=1}^u a_kr_k)^2\geq (\sum_{i=1}^na_ir_i)^2.
     \end{equation}
We first observe   that, by the Cauchy-Schwarz inequality and equations \eqref {e: a_i r_i}, \eqref{e: reduced}, we have
\begin{equation} \label{e: CS}  (\sum_{k=1}^ua_k r_k)^2 \leq  \sum_{k=1}^u(r_k)^2 \times \sum_{k=1}^u(a_k)^2 \leq \sum_{k=1}^u(r_k)^2.
\end{equation}
Then we do the estimate:
\begin{align}\label{e: key estimate}
\sum_{i=1}^n r_i(r_i-1) =\sum_{i=1}^n r_i^2-\sum_{i=1}^n r_i
= \sum_{k=1}^u r_k^2- \sum_{k=1}^u r_k+ (\sum_{l=u+1}^n  r_l^2-\sum_{l=u+1}^n  r_l) \nonumber\\
\geq \sum_{k=1}^u r_k^2-\sum_{k=1}^u r_k \hbox{ \quad (since $x^2-x\geq 0$ for any integer)}   \nonumber\\
\geq (\sum_{k=1}^u a_kr_k)^2 - \sum_{k=1}^u a_kr_k
\hbox{\quad (follows from the two inequalities:}  \\
\hbox{$ -\sum_{k=1}^u  r_k > - \sum_{k=1}^u a_kr_k$ and $\sum_{k=1}^u  r_k^2\geq (\sum_{k=1}^u  a_kr_k)^2)$} \nonumber  \\
\geq (\sum_{i=1}^na_ir_i)^2 -  \sum_{i=1}^n a_ir_i \hbox{ \quad (this crucial step follows from \eqref{e: crucial})}  \nonumber\\
>|p|+|q|+(|p|+|q|)^2.\quad \quad \quad \quad \quad \quad \quad \quad \nonumber
\end{align}
Because  $\sum_{i=1}^n r_i(r_i-1)$ is an integer, we actually  have
$$\sum_{i=1}^n r_i(r_i-1)\geq 1+|p|+|q|+(|p|+|q|)^2.$$
Now the inequality \eqref{e: g, p, q} follows from the inequalities \eqref{e: key estimate} and  \eqref{adj},
 $$0\leq 2g_{\w}(A) =2pq -2(p+q)+2 -\sum_{i=1}^n r_i(r_i-1) \leq |p|+|q|+1-(p^2+q^2).$$

With the  inequality \eqref{e: g, p, q} established, we note that  a direct consequence  is that it is impossible to have $p\leq -2, q\leq 0$,  or $p\leq 0, q\leq -2$:
If $|p|>1,$ $|p|+|q|+1-(p^2+q^2)$ is clearly negative since $q^2\geq |q|, p^2 >|p|+1$; it is the same if $|q|>1$.

So  the inequality \eqref{e: g, p, q} leaves   two cases to analyze:  $p=q=-1$, or $p=-1, q=0$.
To deal with these two cases, as in the proof of  the  inequality \eqref{e: g, p, q}, we assume that $r_k<0$ for $1\leq k\leq u$ and $r_l\geq 0$ for $u+1\leq l\leq n$.
Notice that $\sum_{k=1}^ur_k^2 -\sum_{k=1}^ur_k \leq \sum_{i=1}^n r_i(r_i-1)$ as shown in \eqref{e: key estimate}.

$\bullet$  $p=-1$ and $q=0$.

In this case,  we have  $2g_{\w}(A)= 4-\sum_{i=1}^n r_i(r_i-1)$ so $\sum_{k=1}^ur_k^2 -\sum_{k=1}^ur_k \leq \sum_{i=1}^n r_i(r_i-1)\leq 4$ by the adjunction inequality \eqref{adj}. By the area inequality \eqref{Area}, we have
$\sum_{k=1}^ur_k < p+q=-1$, and hence  $\sum_{k=1}^ur_k \leq-1$.
It is easy to see that  $\{r_k\}=\{-1\}$ or $\{-1,-1\}$. But these possibilities are excluded
by the reduced condition $a_i+a_j\leq 1\leq \mu$ for any pair $i, j$ and the area inequality \eqref{Area}.

$\bullet$   $p=q=-1$.

In this case, we have $2g_{\w}(A)= 8-\sum_{i=1}^n r_i(r_i-1)$ so $\sum_{k=1}^ur_k^2 -\sum_{k=1}^ur_k \leq 8$. By the area inequality \eqref{Area}, we have
$\sum_{k=1}^ur_k < p+q=-2$, and hence  $\sum_{k=1}^ur_k \leq-2$.
It is easy to see that  $\{r_k\}=\{-1,-1,-1\} ,\{-1,-1,-1,-1\}$ or $\{-1,-2\}$.
Again these possibilities are excluded by the reduced condition $a_i+a_j\leq 1\leq \mu$ for any pair $i, j$ and the area inequality \eqref{Area}.


\noindent {\bf Case (v)}. $p=0$.

In this case, by the adjunction inequality \eqref{adj} we have $ -2(q-1)-\sum_{i=1}^n r_i(r_i-1)\geq0$. Since $-\sum_{i=1}^n r_i(r_i-1)\leq 0$, we have   $q\leq 1$.
If $p=0, q\leq 0$, then we apply the inequality \eqref{e: g, p, q} to conclude that $1+|q|-q^2\geq 0$. Since $q\leq 0$, this leaves only the possibilities that $q=0$ or $q=-1$.
We exclude the case $q=-1$.
If $q=-1$, then  by the adjunction inequality \eqref{adj} again we have $ 4=-2(q-1)\geq \sum_{i=1}^n r_i(r_i-1)$. So $-1\leq r_i\leq 2$ for each $i$ and at most two $r_i$ are negative.
By the area inequality \eqref{Area}  we have $-\sum_{i=1}^na_ir_i>-q=1$. But this contradicts with  the reduced condition $a_i+a_j\leq 1\leq \mu$ for any pair $i, j$.
Hence we must have $q=0$ or $1$ if  $p=0$.

  In conclusion, only cases (i),(ii),(v) are possible. Moreover,  if $p=0$, then $q=0$ or $1$;  if $p=1$, then $r_i \in \{ 0,1 \}$; if $p>1$, then $q\geq 1$.

\end{proof}

For rational curves, we further have

\begin{lma}\label{sphere}
Suppose $A=pB+qF-\sum r_iE'_i $
has a simple $J$-holomorphic rational curve representative and $A\cdot A<0$. Then $p=0$ or $1$.
\end{lma}

\begin{proof}
Let us assume  $p\geq 2$ and draw a contradiction. First observe that
  $q\geq 1$ by Lemma \ref{classes}.

 Observe also that,  by the adjunction inequality \eqref{adj},    we have
 \begin{equation}\label{e: g=0 adj}
 \sum_{i=1}^n r_i(r_i-1) = 2(p-1)(q-1).
 \end{equation}
 Since $g_{\omega}(A)=0$ (by Condition 1), $2g_{\omega}(A)-2=K_{\omega}\cdot A+A\cdot A$ and $A\cdot A<0$,  we have $-1\leq K_{\omega}\cdot A =\sum_{i=1}^n r_i -2p-2q$.
Namely,
\begin{equation} \label{e: g=0, r_i}
\sum_{i=1}^n r_i = 2p+2q + k, \quad k\geq -1.
\end{equation}
Now if $p>1, q \geq 1$, since $n \leq 4$, by the   Cauchy-Schwartz inequality and \eqref{e: g=0, r_i},
\begin{equation}\label{e: g=0, r_i r_i}
\sum_{i=1}^n r_i^2 \geq  [\sum_{i=1}^n r_i]^2  /4 \geq (2p+2q+k)^2/4.
\end{equation}
 It follows  from    \eqref{e: g=0, r_i}  and \eqref{e: g=0, r_i r_i} that

 \begin{align} \sum_{i=1}^n r_i^2- \sum_{i=1}^n r_i
 &\geq (2p+2q+k)^2/4 -(2p+2q+k)\nonumber\\
 &=(p+q)^2+(p+q)k + \frac{k^2}{4} -k -2(p+q)\nonumber\\
&= [2pq+2 -2(p+q)]+(p^2+pk-2) + (q^2 +qk-k) +\frac{k^2}{4} \nonumber
\end{align}
Since $p\geq 2,  q\geq 1,  k\geq -1$,  the last 3 terms are all non-negative, and $(q^2 +qk-k)$ is always strictly positive.
To see that $(q^2 +qk-k)$ is always strictly positive, we separate into 2 cases: $k\geq 0$ and $k=-1$. If $k\geq 0$, we have $q^2+qk-k=q^2+k(q-1)\geq q^2\geq 1$. If $k=-1$, we have
$q^2+qk-k=q(q+k)-k=q(q-1)+1\geq 1$.

Therefore we have
$\sum_{i=1}^n r_i^2- \sum_{i=1}^n r_i  >  2(p-1)(q-1)$,
contradicting  with \eqref{e: g=0 adj}.

\end{proof}



\subsubsection{The classification result}

We can explicitly write down all the classes in $\mathcal S_{\w}^{<0}$ for a reduced symplectic form $\w$.

\begin{prp} \label{negative sphere}  Any class in  $\mathcal S_{\w}^{<0}$ lies in one  of the following three disjoint  sets:
\begin{align} \mB&=\{B-lF-\sum r_i E'_i, l\geq-1, r_i\in\{0,1\}\};\nonumber\cr
 \mF&=\{F-\sum r_i E'_i, r_i\in\{0,1\}\};\nonumber\cr
 \mE&=\{E'_j-\sum r_i E'_i, j<i, r_i\in\{0,1\}\}.\nonumber
 \end{align}
In particular, $\mathcal S_{\w}^{<0}$ is a finite set.

 Moreover, a class in $\mB, \mF, \mE$ is in  $\mathcal S_{\w}^{<0}$ if and only if it has positive $\omega$-area.
 $\mS_\w^{<0}$

\end{prp}

\begin{proof} We first show that  $\mathcal S_{\w}^{<0}$ is contained in the union of $\mB, \mF, \mE$.
Suppose $A=pB+qF-\sum_i r_iE'_i$ is such a class.
By Lemma \ref{sphere}, $p=0$ or $1$.

$\bullet$ $p=1$.

If $p=1$, then $r_i=0$ or $1$ as shown in  Lemma \ref{classes}.
So
 $$A= B+qF-\sum r_i E'_i, r_i\in\{0,1\}.$$
 And the condition $A\cdot A<0$ and $n\leq 4$ implies that $q\leq 1$.

 $\bullet$ $p=0$.

In this case,    we have shown that $q=0$ or $1$  in Lemma \ref{classes}.

If $p=0, q= 0$,
 the adjunction inequality \eqref{adj}  is   $2-\sum_{i=1}^n r_i(r_i-1) \geq 0$.
Let $x$ be an integer.
Notice that $x(x-1)\geq 0$, and $x(x-1)=0$ if $x=0$ or $1$. Notice also that since $x$ is an integer, pif $x(x-1)>0$ then
$x(x-1)\geq 2$, and $x(x-1)=2$ if $x=2$ or $x=-1$. We see there is at most one $j$ such that $r_j\ne 0$ or $\ne 1$,
and for this $j$, $r_j=-1$ or $2$.   By considering  the area of such a class, we must have   $r_j=-1$, and $j<i$ for any
$r_i=1$.
Therefore, in this case,
$A$  can only be of the form
$$E'_j-\sum r_iE'_i, i>j, \quad r_i\in \{ 0,1 \}.$$

We are left with  $q=1$.  In this case,  the adjunction inequality \eqref{adj}  is of the form  $-\sum_{i=1}^n r_i(r_i-1) \geq 0$.
So we must have $r_i=0$ or $1$. Namely,
 $$A=F-\sum r_i E'_i, \quad r_i\in\{0,1\}.$$

 Now we show that $\mathcal S_{\w}^{<0}$ is a finite set. By the first statement, it suffices to show that the intersection of
 $\mathcal S_{\w}^{<0}$ with $\mB$, $\mF$, or $\mE$ is a finite set.
 The sets $\mF$ and $\mE$ are clearly finite since $r_i=0$ or $1$. To show that  the set $\mB\cap \mathcal S_{\w}^{<0}$ is finite for a reduced $\w$ it suffices to show that $[\w]$ pairs positively with finitely many elements in $\mB$.
  A reduce  $\w$ has class
 $[\w]=B+\mu F-\sum  a_i E_i'$ satisfying \eqref{reducedBF}, and hence
 the pairing between $[\w]$ and a class in $\mB$ is $\mu-l-\sum a_i r_i\leq mu-l$, which is negative if $l>\mu$.

 To show that  a class $A$  in $\mB, \mF, \mE$ is in  $\mathcal S_{\w}^{<0}$ if  it has positive $\omega$-area, we need to construct an $\omega$-symplectic sphere in the class $A$.
 Notice that we can get an embedded holomorphic rational curve in the class $A$
via small K\" ahler blowup.
Then we apply Theorem 2.11  in \cite{DoL10}.

\end{proof}


For convenience,  we list the set $R^+_{n+1}=\mB^{-2}\coprod \mF^{-2} \coprod \mE^{-2}$  for $S^2\times S^2  \# n{\overline {\CC P^2}}$ (see Proposition  \ref{MLS}).
\begin{itemize}
\item
$R^+_{1}=\{ B-F \}$.
\item
$ R^+_{2}=\{ B-F\}$.
\item
 $R^+_{3}=\{  B-E'_1-E'_2,\, B-F, \,F-E'_1-E'_2, \,E'_1-E'_2 \}$.
\item
$R^+_{4}=\{  B-E'_i-E'_j, \, B-F, \, F-E'_i-E'_j, \, E'_j-E'_i, \, 1\leq j<i\leq 3\}$.
\item
$ R^+_{5}=\{ B+F-E'_1-E'_2-E'_3-E'_4,\, B-E'_i-E'_j,\, B-F, \, F-E'_i-E'_j,\,
E'_j-E'_i, \, 1\leq j< i\leq 4\}$.
\end{itemize}

We remark that  Proposition \ref{negative sphere}  overlaps with and can be derived from results in Section 4.1 in  \cite{Zha17}, which are in a slightly different context. In \cite{Zha17}, while the almost complex structure is assumed to have the standard canonical class, the symplectic structure is not assumed to be reduced.

 \begin{rmk}\label{toric-2}
   The following observation (coming from toric moment polytope, also carefully written in \cite{AP13} section 4.2 and 4.3) will be
   used in Lemma \ref{injection}: for a symplectic rational surface $X$ with  $\chi(X)\leq 7$, any
   $A\in \mS_{\w}^{-1}\cup \mS_{\w}^{-2}$ arises as
   an edge of a toric moment polytope. Consequently, there is a semi-free circle action having a symplectic sphere $S$
   as a component of the fixed loci with $[S]=A$. See also Remark \ref{s1gen}.
\end{rmk}

\subsection{Cusp decomposition and level 2 stratification}
We continue to assume that $\chi(X)\leq 8$ and $\w$ is reduced.

\subsubsection{Cusp curve decompositions of a class in  \texorpdfstring{$\mathcal S_{\w}^{<0}$}{Lg}}


Here is an important consequence of  Lemma \ref{classes} and    Proposition \ref{sphere}.

\begin{prp}\label{decp}
For any stable curve  in a class $A \in \mathcal S_{\w}^{\leq -2}$,

$\bullet$ there are no components with positive self-intersection,

$\bullet$ any self-intersection  zero component   is in the class  $B$ or $kF, k>0$,

$\bullet$ any simple component with negative self-intersection is embedded and hence its class is in $\mathcal S_{\w}^{<0}$.

\end{prp}

\begin{proof}
Suppose $A =pB+qF-\sum_i r_iE'_i \in \mathcal S_{\w}^{\leq -2}$.
Then by Lemma \ref{sphere}, $p=0  \hbox{ or }  p= 1$,  and if $p=1$ then $q\leq 1$.
We argue by contradiction to show that there are no components with positive self-intersection in any cusp curve decomposition of $A$.
Assume   the cusp  decomposition \eqref{cusp curve equation} of $A$  has the following form for some $J$,  $$A=[C']+ \sum_{\alpha}m_{\alpha}[C_{\alpha}],$$
 where $C'$  is a simple $J$-holomorphic rational curve with  positive self-intersection and $[C']=p'B+q'F-\sum_i r'_iE'_i$,
$m_{\alpha}>0$ and each $C_{\alpha}$ is a simple $J$-holomorphic rational curve. Notice that it is possible that $C'=C_{\alpha}$ for some $\alpha$.

 By Lemma \ref{classes},  the $B$ coefficients of $A$ (the number $p$),  $[C']$ (the number $p'$) or any $ [C_{\alpha}]$ (denoted by $p_{\alpha}$) are all  non-negative. Moreover, since $A$ admit some embedded J-holomorphic representative, $A\cdot A<0$ and $g_{\w}(A)=0$,
 by Lemma \ref{sphere} (or Proposition \ref{negative sphere}), $p=0$ or $1$.
 Since $[C']\cdot [C']>0$, we have $p'> 0$. So we have $p=\sum_{\alpha} p_{\alpha}+p'\geq p' > 0$.
 Since $p=0$ or $1$, we must have   $p=p'=1,  p_{\alpha}=0, \forall \alpha$.

 Now let us inspect the $F$ coefficients $q, q', q_{\alpha} $. First of all, we have  $q\leq 1$, by Proposition \ref{negative sphere}.
 For the class $[C']$, since $[C']\cdot [C']>0$ we have
  $p'q'=q'\geq 1$.  For any $ [C_{\alpha}]$ class, since the $B$ coefficient $p_{\alpha}$ is zero, by Lemma \ref{classes} the $F$ coefficient $q_{\alpha}$ is $0$ or $1$.
 Hence $q=q'+\sum_{\alpha}q_{\alpha}\geq q'\geq 1$. Since we have also observed that $q\leq 1$,    we conclude that both  $q=1$ and $q'=1$, $q_{\alpha}=0, \forall \alpha$.

 Since $p=q=1, A\cdot A\leq -2$, we must have $A=B+F-E'_1-E'_2-E'_3-E'_4$ from Proposition \ref{negative sphere}. In addition, since $p'=q'=1$,
 by Lemma \ref{classes},  $[C']=B+F-\sum_i r'_iE'_i$, $r'_i\in\{0,1\}$.
However, the sum of the curve class   $\sum_{\alpha}m_{\alpha}[C_{\alpha}] = A- [C']= -E'_1-E'_2-E'_3-E'_4 +\sum_i r'_iE'_i$, $r'_i\in\{0,1\}, 1\leq i \leq 4,$  has negative symplectic area. Contradiction!
Hence there are no positive self-intersection components
 in a cusp curve decomposition \eqref{cusp curve equation} of a class $A \in \mathcal S_{\w}^{\leq -2}$.

Next, we analyze the possible square 0 classes in the decomposition. From the analysis above, we only need to deal with the case that either $p'=0$ or $q'=0$.
For the case $p'=0$, the only possible square zero classes are  $kF, k\in \ZZ^+$. And for the case $q'=0$, the only possible square zero class is $B$.

The last bullet simple component with negative self-intersection follows from  Condition \ref{negsphemb}.
\end{proof}

\begin{rmk}\label{overlap}
Note that the following results  overlap with \cite{Zha17}: Lemma 4.1 and Proposition 4.6 in \cite{Zha17} cover Theorem \ref{negative sphere} of this paper up to Cremona transformations; Lemma 4.12 in \cite{Zha17} also provided useful information about \eqref{cusp curve} for small rational surfaces, and will be useful in \cite{L1} for the general stratification.

\end{rmk}




\subsubsection{Level 2 stratification}
We are ready to prove  Theorem \ref{level 2}, which we restate here.

\begin{thm}\label{rational}
For a symplectic rational surface $(X, \w)$ with  $\chi(X)\leq 8$,
$\mX_4=\amalg_{cod(\mC)\geq4} \mJ_{\mC}$ and $\mX_2=\amalg_{cod(\mC)\geq2} \mJ_{\mC},$
 are closed subsets  in $\mX_0=\mJ_{\w}$. Consequently,

(i).  $\mX_0 -\mX_4$ is a manifold.

(ii).  $\mX_2 - \mX_4$  is closed in  $\mX_0 - \mX_4$.

(iii). $\mX_2 -\mX_4$ is a manifold.

(iv).  $\mX_2-\mX_4$ is a submanifold of  $\mX_0 -\mX_4$.

\end{thm}

\begin{proof}
We can assume that $\omega$ is reduced by Theorem    \ref{redtran}.
Let us recall that,  by Lemma \ref{rat1},  Condition \ref{negsphemb} applies here.

We first  show that $\mX_2$ is  closed  in $\mX_0$, namely,  $\overline{\mX_2} \cap (\mX_0-\mX_2) = \emptyset$.
We will argue by contradiction.

Suppose   $\overline{\mX_2} \cap (\mX_0-\mX_2) \ne \emptyset$.  Then there is a sequence  $ \{J_n\}$ in $\mX_2$ converging to $ J'\in (\mX_0-\mX_2)=\mJ_{\emptyset}.$
Since $\mS_{\w}^{<0}$ is finite by Proposition \ref{negative sphere}, which means there are only finitely many labeling sets $\mC$ for which $\mJ_{\mC}$ is nonempty. Thus we can assume that  the sequence  $\{J_n\} $ is in a single prime set $\mJ_{\mC}$ and we can apply Lemma \ref{l:dec} here.
By Condition \ref{negsphemb},  for $J' \in \mX_0-\mX_2$, every simple $J'$-holomorphic rational curve has self-intersection  at least $-1$ and the ones with self-intersection $-1$ are embedded.
Suppose $A\in \mC$. Then for each $J_n\in \mJ_{\mC}\subset  \mX_2$ there is  one embedded $J_n$-holomorphic rational curve in the class $A$.
By Lemma \ref{l:dec} the class $A\in  \mS_{\w}^{\leq -2}$   admits a decomposition  as in equation \eqref{cusp curve}, with no class having square less than  $-1$.
Moreover, by Proposition  \ref{decp},
the decompositions can be written as  $$A=rB+ \sum_i a_iF+ \sum_j b_jD_j,$$ where $a_i, b_j$ are non-negative integers, $r \in \{0,1\},$ and $D_j\in \mathcal S^{-1}_{\omega}$ by Condition 1.
By pairing with $K_{\omega}$ on both sides. The left hand side is
 $ A\cdot K_{\omega}\geq 0$, and the right hand side is $r B \cdot K_{\omega}+ \sum a_i F\cdot K_{\omega} +\sum_j b_j D_i \cdot K_{\omega}\leq 0,$ since $ D_i \cdot K_{\omega} =-1, F \cdot K_{\omega}=B \cdot K_{\omega}=-2$.  This is a contradiction.

We next  show that $\mX_4$ is closed in $\mX_0$, namely,
$\overline{\mX_4} \cap (\mX_0-\mX_4)= \emptyset.$
Since $\mX_4\subset \mX_2$ and $\mX_2$ is closed in $\mX$, it suffices to show that $\mX_4$ is closed in $\mX_2$. The argument is similar.
By Condition \ref{negsphemb}, for each $J' \in \mX_2-\mX_4$ every simple $J'$-holomorphic rational curve has self-intersection at least $-2$, there is exactly one simple $J'$-holomorphic rational curve with self-intersection  $-2$,
and the ones with self-intersection $-1$ and $-2$
are embedded.
For each $J_n\in \mX_4$ there is either  one embedded $J_n$-holomorphic rational curve with self-intersection at most $-3$, or there are at least two embedded $J_n$-holomorphic rational curves  with self-intersection $-2$.

Let us   assume that  $\overline{\mX_4} \cap (\mX_2-\mX_4) \ne \emptyset$, and there is  a sequence  $\{J_n\}$  in $\mX_4$   converging to $J'\in \mJ_{\mC'}\subset (\mX_2-\mX_4)$   for some labeling set $\mC'$.  Notice that $\mC'$ is of the form $\{A'\}$ for some $A'\in \mS_\w^{-2}$.

Again, because of the finiteness of labeling sets, we can assume $\{J_n\}$ all lie in $\mJ_{\mC}$ for some labeling set $\mC$.
So we are again in the situation of Lemma \ref{l:dec}.

\noindent 1) Assume first $\mC$ contains a class $A\in \mS_{\w}^{\leq -3}$.
Then by Proposition  \ref{decp},
there is a cusp curve decomposition  $A= c'A'+ \sum c_i A_i$ with $A_i \in   \mS_{\w}^{-1} \coprod \mS_{\w}^{0}$, $c'\geq 0, c_i\geq 1$.
Pair the cusp curve decompositions with $K_{\w}$.  We get a contradiction since $K_{\w}\cdot A>0$ while $K_{\w}\cdot A'=0, K_{\w}\cdot A_i< 0$.

\noindent 2) Otherwise, $\mC$ contains  at least two   classes in $\mS_{\w}^{ -2}$ and one of them has to be distinct from $A'$. Call this class $A$.    By Proposition  \ref{decp},
there is  a cusp curve decomposition $A= c'A' +\sum c_i A_i$ with  $\{A_i\} \subset  \mS_{\w}^{-1} \coprod \mS_{\w}^{0}, c'\geq 0, c_i\geq 1$.
We claim that the set $\{A_i\}$ is not empty. This is true because $A\ne A'$ and  $c'A'$ is not in $S_{\w}^{-2}$ for any  $c'\ne 1$. Now
 pair the cusp curve decompositions with $K_{\w}$.  We again get a contradiction since  $K_{\w}\cdot A=0$ while $K_{\w}\cdot A'=0$ and
$K_{\w}\cdot A_i<0$.

So we have proved that $\mX_2$ and $\mX_4$ are closed in $\mJ_{\w}.$\\

Next, let us establish the claims (i)-(iv).

(i). $\mX_0 -\mX_4$ is a manifold.  This statement is true since  $\mX_4$ is closed in  $\mX_0$ and $\mX_0$ is a manifold. Similarly, $\mX_0-\mX_2$ is a manifold since $\mX_2$ is also closed in $\mX_0$. And both
$\mX_0-\mX_4$ and $\mX_0-\mX_2$ are open submanifolds of $\mX_0$.

(ii). $\mX_2 -  \mX_4$  is closed in  $\mX_0 - \mX_4$.    This follows from the fact that   $\mX_2$ is closed in  $\mX_0$.

(iii). $\mX_2 -\mX_4$ is a manifold.
This statement follows from   the fact that  $\mX_2-\mX_4$ is a submanifold of $\mX_0$. This latter fact follows from Proposition \ref{submfld} and the fact that $\mX_2 - \mX_4 $ is the disjoint union of codimension  2  prime sets $\mathcal J_{A}$ over $A\in \mathcal S_{\omega}^{-2}$.

(iv). $\mX_2-\mX_4$ is a closed submanifold of  $\mX_0 -\mX_4$.    Since $\mX_0-\mX_4$ is an open submanifold of $\mX_0$,  $\mX_2-\mX_4$ is also   a submanifold of $\mX_0-\mX_4$.



Hence this proves that $\mX_{4}(=\mX_{3}) \subset \mX_2 (= \mX_1) \subset \mX_0 = \mJ_{\omega},$
is a level 2 stratification.
\end{proof}

It means that   the decomposition
 $\mJ_{\omega}$ is a stratification at the first two levels with top stratum $\mJ_{open}$.


\begin{rmk} \label{stratification}
In a separate paper \cite{L1} the first author will further show that, for a symplectic rational surface $(X, \w)$ with $\chi(X)\leq 8$,
  this filtration of $\mJ_{\omega}$ fits into the following notion of   {\bf even stratification}
 in the  $\infty$-dimensional setting (For finite dimension, see eg.  \cite{GM88}):

\begin{dfn}\label{str}
For an $\infty $-dimensional real Fr\'echet manifold $\mX$, a finite filtration
 of $\mX$ is called an  even stratification if it is a sequence of {\bf closed} subspaces
$$ \mX_{2n+2}:=\emptyset  \subset \mX_{2n} \subset \mX_{2n-2} \ldots \subset \mX_2 \subset \mX_0 = \mX ,$$
where $\mX_{2i}- \mX_{2i+2}$ is a submanifold $\mX_0$ of real codimension $2i$ for $0\leq i\leq n$.
\end{dfn}

\end{rmk}

 There is a decomposition of $\mJ_{\omega}$   for
   $({\CC P^2}\# 3 \overline{\CC P^2},\omega)$  in Lemma 2.10 of \cite{AP13}
   by the existence of a certain negative self-intersection curve.
  They have also shown that this decomposition is a stratification with finite-codimensional
  submanifolds as strata.
Our decomposition is finer in the sense that each stratum
    in \cite{AP13} is a union of prime submanifolds in our decomposition. In particular,
     our decomposition being a stratification as in Definition \ref{str} implies their decomposition
     is a stratification.

\subsection{Symplectic \texorpdfstring{$(-2)$-}{Lg}spheres and \texorpdfstring{$H_1(\mJ_{open})$}{Lg} }

\subsubsection{Enumerating the components of \texorpdfstring{$\mX_2-\mX_4$ by $\mathcal S_{\w}^{-2} $}{Lg}}\label{-2}
Let $N_{\w}$ denote the cardinality of $\mathcal S_{\w}^{-2}$. We have the following  crucial result about codimension 2 strata in the stratification of $\mJ_{\omega}$.

\begin{prp}\label{conn}
  For $X=S^2\times S^2 \# n\overline{\CC P^2}, n=0,1,2,3$,
 $\mJ_{A}$ is path connected if $A\in \mS_\w^{-2}$. Hence the number of  path connected components in  $\mX_2-\mX_4$ is $N_{\w}$.
\end{prp}

We need the following lemma, which was stated   in \cite{LLW15}:
\begin{lma}\label{tran}
 For a symplectic rational manifold $(X, \omega)$ with $X=S^2\times S^2 \# n\overline{\CC P^2}, n\geq 0$,
 the group $Symp_h(X, \omega)$ acts
 transitively on the space of homologous $(-2)$-symplectic spheres.
 \end{lma}
 \begin{proof}
 Here we give a  proof  following  steps sketched in   \cite{LW12} and \cite{BLW12}.
 Let $S_1, S_2$ be two homologous symplectic $(-2)$-spheres.
  Without loss of generality, we can assume the symplectic sphere $S_i$ is in the
  homology class $[S_i]=B-F,$ since we can always change basis in $H_2(X,\RR).$  For each  $(X, \w, S_i)$, by \cite{MO15}, there
  is  a set $\mC_i$ of disjoint $(-1)$-symplectic spheres $C_i^l$ for $l=1,\cdots, k$ such that
  $$ [C_i^l]=E'_l, \text{for}\quad l= 1,\cdots n. $$
  Blowing down the set $\{ C_i^1, \cdots C_i^n\}$ separately, one obtains a $4-$tuple $(X_i, \w_i, \tilde{S_i},\mB_i)$,
  where  $(X_i,\w_i)$ is are symplectic $S^2\times S^2$  with $[\w_1]=[\w_2]$,
  $\tilde{S_i}$ a symplectic sphere in  $X_i$,
  and $\mB_i=\{B_i^1,\cdots,  B_i^n\}$ is a symplectic ball packing in $X_i\setminus \tilde{S_i}$
  corresponding to $\mC_i$.
For any two such tuples  $(X_1, \w_1, \tilde{S_1},\mB_1)$ and  $(X_2, \w_2, \tilde{S_2},\mB_2)$, since the symplectic forms are homologous, by \cite{LM96},
there is a symplectomorphism $\Phi$
from $(X_1,\w_1,  \tilde{S_1})$ to $(X_2, \w_2, \tilde{S_2}),$  such that for a fixed $l$,
$Vol(\Phi(B_1^l))=Vol(B_2^l)$.
Then according to \cite{AM00}, we
can choose this $\Phi$ such that the two symplectic spheres are
isotopic, i.e. $\Phi(\tilde{S_1}) = \tilde{S_2}$. Applying Theorem 1.1 in  \cite{BLW12}, there is a compactly supported
Hamiltonian isotopy $\iota$ of $(X_2, \w_2, \tilde{S_2})$ such that the symplectic ball packings
$\Phi(\mB_1)$ and $\mB_2$ are connected by $\iota$ in $X_2\setminus \tilde{S_2}.$ Then
$\iota \circ \Phi$ is a symplecotomorphism between the tuples $(X_i, \w_i, \tilde{S_i},\mB_i)$
and hence blowing up induces a symplecotomorphism $\psi: (X_1,\w_1,  \tilde{S_1},\mB_1) \rightarrow
(X_2, \w_2, \tilde{S_2},\mB_2). $  Further note  that $\psi$ preserves homology classes
$B,F, E'_1, E'_2, \cdots, E'_n$ and hence $\psi \in Symp_h(X,\w).$
 \end{proof}

\begin{proof}[Proof of Proposition \ref{conn}]
Applying Theorem 1.1 in \cite{LLW15}, $Symp_h(X, \omega)$ is itself path connected  for $X=S^2\times S^2 \# n\overline{\CC P^2}, n=0,1,2,3$.   Therefore, for $A\in \mS_\w^{-2}$, the space   of  symplectic $(-2)$-spheres  in the class $A$ is path connected
by  Lemma \ref{tran}.  Let  $\mathcal Z_A$ denote the space   of  symplectic $(-2)$-spheres  in the class $A$.
 Since the space of almost complex structures making a symplectic sphere pseudo-holomorphic is weakly contractible (\cite{Eva11}),  and there is a unique embedded
 $J-$holomorphic sphere in the class $A$  for each $J\in \mJ_{A}$,
 $\mJ_{A}$ fibers over  $\mathcal Z_A$ with weakly contractible fibers.  Hence $\mJ_{A}$ is weakly homotopic to  $\mathcal Z_A$. In particular, $\mJ_{A}$ is path connected.

The last claim now follows from the disjoint union decomposition $\mX_2-\mX_4=\coprod_{A\in \mS_\w^{-2}} \mJ_{A}$.
\end{proof}

\subsubsection{Relative Alexander-Pontrjagin duality for regular Fr\'echet stratification}
Let $\mX$ be a Hausdorff space,  $ \mZ \subset \mY $  a closed subset of
$\mX$ such that $\mX-\mZ, \mY-\mZ$ are paracompact manifolds modeled by topological linear spaces.  Suppose $\mY-\mZ$ is a closed co-oriented submanifold of $\mX-\mZ$ of  codimension $p$, then we say $(\mY,\mZ)$ is a closed relative submanifold of $(\mX,\mZ)$ of codimension $p$.
 We have the following relative version of  Alexander-Pontrjagin duality  in \cite{Eells61} when taking constant coefficient:


\begin{thm}\label{alex}
   In the above situation,  we have an isomorphism $H^i(\mX-\mZ,\mX-\mY; G) \cong H^{i-p}(\mY-\mZ; G)$ for  any Abelian coefficient group $G$.

 \end{thm}

Now we restate and prove Corollary \ref{Jopen} here:

\begin{cor}
For  a symplectic rational surface with  $\chi(X)\leq 8$ and any Abelian group $G$,
$H^1(\mJ_{open}; G)= \oplus_{A_i \in \mathcal S_{\omega}^{-2}} H^0(\mJ_{A_i}; G)$.

If we further assume that $\chi(X)\leq 7$, then for  each  $ {A_i \in \mathcal S_{\omega}^{-2}}  $,  $\mJ_{A_i}$ is path connected and hence $H^1(\mJ_{open};G)=G^{N_{w}}$, where $N_{\omega}$ is the cardinality of $\mathcal S_{\omega}^{-2}$.
It follows from the universal coefficient theorem that   $H_1(\mJ_{open}; \ZZ)=\ZZ^{N_{\w}}$.
\end{cor}
\begin{proof}

For  $ \CC P^2 $, $\mJ_{open}=\mJ_{\w}$ and $N_{\w}=0$ so the statement holds trivially.

For any other $X$ with $\chi(X)\leq 8$,  let $\mX=\mX_0, \mY=\mX_2, \mZ=\mX_4 $. By Proposition \ref{submfld} and Theorem \ref{rational},   the conditions  in
Theorem \ref{alex} are satisfied.
Consider the cohomology long exact sequence  for the pair  $(\mX-\mZ,\mX-\mY)$ and replace  $H^i(\mX-\mZ,\mX-\mY;G)$  by $H^{i-p}(\mY-\mZ;G)$ by Theorem \ref{alex}. Then we have the long exact sequence
\[
\dotsb \rightarrow H^{i-1}(\mX-\mY; G)  \rightarrow H^{i-p}(\mY-\mZ;G)  \rightarrow H^{i }(\mX-\mZ;G)  \rightarrow H^{i}(\mX-\mY;G)  \rightarrow  \dotsb
\]
  Since $\mX=\mX_0$ is contractible, $\mZ=\mX_4$ is a union of submanifolds of  codimension 4 or higher. Note that by Remark \ref{invlim} (the same convention as in \cite{AGK09}), each $\mX_i$ is the inverse limit of its  Banach counterparts $\{\mX_i^p\}$, where  $\mX_0^p$ is contractible as well and $\mZ^p=\mX_4^p$ has codimension 4 or higher in $\mX_0^p$. We have  transversality in the Banach setting (see Chapter 2, page 27-28 of \cite{Lang99}) and hence any  map into $\mX_0^p-\mX_4^p$ with domain being $S^1$ or $S^2$ is homotopic to a constant map  since  $\mX_0^p$ is contractible and $\mZ^p=\mX_4^p$ has codimension 4 or higher in $\mX_0^p$. Then we know that $\pi_1(\mX^p-\mZ^p)=\pi_2(\mX^p-\mZ^p)=0$ in each Banach setting. Hence by the Hurewicz Theorem, $H_1(\mX^p-\mZ^p,\ZZ)=H_2(\mX^p-\mZ^p,\ZZ)=0$ for each $p$.   Then by the universal coefficient Theorem,  for any abelian group $G$,  $H^{1 }(\mX^p-\mZ^p ;G) =H^{2 }(\mX^p-\mZ^p;G )=0$ for each $p$. Then taking the inverse limit, we have  $H^{1 }(\mX-\mZ ;G) =H^{2}(\mX-\mZ;G )=0$ in the Fr\'echet setting by Remark \ref{invlim}.

 Setting $i=p=2$, we have the exact sequence
\[
0=H^{1 }(\mX-\mZ ;G) \rightarrow H^{1}(\mX-\mY;G)  \rightarrow H^{0}(\mY-\mZ;G)  \rightarrow H^{2 }(\mX-\mZ;G )=0,
\]
namely,  $H^{1}(\mX-\mY;G) =H^{0}(\mY-\mZ;G) .$ Notice that  $\mX-\mY=\mJ_{open}$ and $\mY-\mZ$ is the disjoint union of $\mJ_ A$ for $A\in \mS_\w^{-2}$, we have
$H^1(\mJ_{open}; G)= \oplus_{A_i \in \mathcal S_{\omega}^{-2}} H^0(\mJ_{A_i}; G)$.

When  $\chi(X)\leq 7,$  the statement $H^1(\mJ_{open};G)=G^{N_{\w}}$  follows from Proposition \ref{conn}.
Then,  by the Universal Coefficient Theorem,
we have $$ G^{N_\w} \cong H^0 (\mX_2-\mX_4;G) \cong H^1(\mX_0-\mX_2 ;G) \cong Hom(H_1(\mX_0-\mX_2;\ZZ);G)$$
 for any Abelian group $G$.   This implies that $H_1(\mX_0-\mX_2;\ZZ) =\ZZ^{N_\w}.$

\end{proof}

\begin{rmk}  \label{invlim} As noted in \cite{AGK09} (Convention part or Remark 2.2),  each Fr\'echet manifold we work with can naturally be interpreted as the inverse limit of a sequence of Banach manifolds. In fact, Proposition \ref{stratum} is valid and established first in the Banach settings.
Moreover,  it is noted in \cite{AGK09} that the successive inclusions between the Banach manifolds are weak homotopy equivalences. Therefore the results about the homology and homotopy groups in the Fr\'echet setting can be interpreted as the corresponding results for each Banach manifold in the sequence.
\end{rmk}

\begin{rmk}

An absolute version of Alexander-Pontrjagin duality in \cite{Eells61} was applied by Abreu in \cite{Abr98} to detect the topology of $\mJ_{open}$ for $S^2\times S^2$ with a symplectic form with ratio within $(1, 2)$.
In \cite{L1},  the first author will establish  an Alexander-Pontrjagin duality for even stratifications as in Definition \ref{str}, generalizing \cite{Eells61}.
The following special case can also be applied to compute the fundamental group of the symplectomorphism group of small rational 4-manifolds:
  Let $\mX_0$ be a contractible  paracompact  Fr\'echet manifold which  is evenly stratified by $\{\mX_{2i}\}_{i=0}^n$ as in
 Definition \ref {str} at the first 2 levels.
  Then we have the following duality on
 the cohomology of $\mX_0 -\mX_2$ and $\mX_2 - \mX_4$, with coefficient $G$ being any Abelian group:
 $$ H^{1}(\mX_0 - \mX_2;G)\cong H^{0}(\mX_2  - \mX_4;G).$$
\end{rmk}

\subsection{Characterizing \texorpdfstring{$\mJ_{open}$}{Lg} via  subsets of \texorpdfstring{$\mathcal S_{\w}^{-1}$}{Lg}}\label{Dopen}

Let $X_{n+1}$ continue to be $S^2\times S^2  \# n{\overline {\CC P^2}}, n\leq 4$.
We next give a characterization of $\mJ_{open}$ using a subset of classes $\mfD'_{n+1} \subset \mS_{\w}^{-1}$. Note that the subset $\mfD'_{n+1}$ here is different from a labelling set $\mC$ in Definition \ref{fine decomposition} since  $\mC\subset \mS_{\w}^{\leq -2}.$
Recall that in  Definition \ref{jd}, for $\mfD \subset \mS_{\w}^{-1}$, $\mJ^{\mfD}$denotes the set of $J$ such that  each class in $\mfD$   has a  $J$-holomorphic representative. Clearly,  if $\mfD \subset \mfC,$ then $ \mJ^{\mfC} \subset \mJ^{\mfD}.$

\begin{lma}\label{-1open}
Let $X_{n+1}$ be  $S^2\times S^2  \# n {\overline {\CC P^2}}, n\leq 4$ with a reduced symplectic form $\w$.
 For the following configuration $\mfD'_{n+1} \subset \mS_{\w}^{-1}(X_{n+1})$,

  \begin{itemize}
\item   $\mfD'_2= \{B-E'_1\},$
\item $\mfD'_3= \{ B-E'_1, F-E'_1, E'_1 \},$
\item
$\mfD'_4= \{  B+F-E'_1-E'_2-E'_3, B-E'_1, F-E'_1, E'_2\},$
\item $\mfD'_5 = \{B+F-E'_2-E'_3-E'_4, B-E'_1,F-E'_1,E'_2,E'_3\},$
 \end{itemize}
 we have $\mJ^{\mfD'_{n+1}}=\mJ_{open}$.
 Consequently, if $\mfC \subset \mS_{\w}^{-1}$  contains $\mfD'_{n+1}$, then  $\mJ^{\mfC}=\mJ_{open}$.
\end{lma}

\begin{proof}
By the 4-th bullet of Lemma \ref{open and -1} we just need to check that any class in $\mS_{\omega}^{\leq -2}$
 intersects at least one class in the configuration $\mfD_{n+1}'$ negatively.
By  Proposition  \ref{negative sphere}, $\mS_{\omega}^{\leq -2}$ is contained in  $\mB \coprod \mF \coprod \mE$. It is straightforward to check each class in $R_{n+1}=\mB^{-2} \coprod\mF^{-2} \coprod \mE^{-2}$ pairs negatively with at least one class  in the configuration $\mfD_{n+1}'$. Note that $R_{n+1}$ is the root system.

For $X_2$ and $X_3$, by Proposition \ref{negative sphere}, any class in $\mB^{\leq -3}$  can be written as $B-qF-r_iE'_i, q\geq1, i\leq k, r_i\in\{0,1\}$,  and
$(B-qF-r_iE'_i)\cdot (B-E'_1)=-q-r_1<0.$

For $X_4$, the only  class in $\mF^{\leq -3}$
 is $F-E'_1-E'_2-E'_3$, and the only class in $\mE^{\leq -3}$  is    $E'_1-E'_2-E'_3$.  Each of them pairs with $B+F-E'_1-E'_2-E'_3$ negatively.
And any class in $\mB^{\leq -3}$  can be written as either $B-E'_1-E'_2-E'_3$ or $B-kF-r_iE'_i, k\geq1, r_i\in\{0,1\}$, by Proposition \ref{negative sphere}. And $(B-E'_1-E'_2-E'_3)\cdot (B+F-E'_1-E'_2-E'_3)<0,$
$(B-kF-r_iE'_i)\cdot(B-E'_1)=-k-r_1<0.$

For $X_5$, $\mF^{\leq -3}=\{ F-E'_1-E'_2-E'_3-E'_4,
 F-E'_i-E'_j-E'_k, 1\leq i<j<k\leq 4 \}$, and
 $\mE^{\leq -3}=\{E'_1-E'_2-E'_3-E'_4, E'_i-E'_j-E'_k, 1\leq i<j<k\leq 4 \}$.
 Each of them pairs with $B+F-E'_2-E'_3-E'_4$ negatively.
And any class in $\mB^{\leq -3}$  can be written as either $B-E'_i-E'_j-E'_k,  1\leq i<j<k\leq 4 ,$    or $B-kF-r_iE'_i, k\geq1, r_i\in\{0,1\}.$ And $(B-E'_i-E'_j-E'_k)\cdot (B+F-E'_2-E'_3-E'_4)<0,$
$(B-kF-r_iE'_i)\cdot (B-E'_1)=-k-r_1<0.$


Therefore any sphere class with square less than $-1$ cannot have a simultaneous  $J$-holomorphic representative with the set $\mfD.$

Finally, the equality $\mJ^{\mfC}=\mJ_{open}$ follows from $\mJ_{open}=\mJ^{\mfD'_{n+1}}\supset  \mJ^{\mfC} \supset \mJ_{open}$, where the first inclusion follows directly  from Definition
\ref{jd} and the last inclusion is the 3rd bullet of Lemma \ref{open and -1} .
\end{proof}

In terms of the reduced basis $\{H, E_1, \cdots, E_{n+1}\}$  $X_{n+1}={\CC P^2} \# (n+1)\overline{ {\CC P^2}}$, we have

\begin{itemize}
\item   $\mfD'_2=\{E_1\}$ for ${\CC P^2} \# 2{\overline {\CC P^2}}$,
\item $\mfD'_3=\{H-E_1-E_2, E_1, E_2 \}$ for ${\CC P^2} \#  3{\overline {\CC P^2}}$,
\item  $\mfD'_4=\{ H-E_3-E_4, E_1, E_2, E_3 \}$ for ${\CC P^2} \#  4{\overline {\CC P^2}}$,
\item
$\mfD'_5=\{ 2H-E_1-E_2-E_3-E_4-E_5,  E_1, E_2,  E_3, E_4 \}$ for ${\CC P^2} \#  5{\overline {\CC P^2}}$.
 \end{itemize}




We mention a couple of  related facts that will not be used in the paper.  First of all, $\mfD'_{n+1}$ is minimal in the sense that, for any {\it proper} subset $\mfC'$ of  $\mfD'_{n+1}$, $\mJ^{\mfC'}$ is strictly bigger than $\mJ_{open}$. Notice that $\mfD'_{n+1}$ does not contain the smallest area class $E_{n+1}$ (cf.  Lemma \ref{minemb}).
Secondly,  there are subsets $\mT \subset \mS_{\w}^{-1}$ not containing $\mfD'_{n+1}$ such that  $\mJ^{\mT}=\mJ_{open}$ .


\section{The rank of  \texorpdfstring{ $\pi_1(Symp_h(X, \omega))$ }{Lg}when  \texorpdfstring{$\chi(X)\leq 7$ }{Lg} }\label{leq4}
In this section,   for a symplectic rational surface $(X, \w)$ with $\chi(X)\leq 7$,   we prove Theorem \ref{rank}, which includes both Theorem \ref{sum} and Corollary \ref {4Q}.
 Note that  $Symp_h(X,\omega)$ is path connected in this case.

As alluded in the introduction, we will treat the cases $\chi(X)\leq 4$ and $5\leq \chi(X) \leq 7$ separately.
  In the former case, we will prove Theorem \ref{rank} by combining   various known computations of $\pi_0(Symp(X, \omega))$  and  $\pi_1(Symp(X, \omega))$ scattered in the literature.
In the latter case we will use  an appropriate configuration of exceptional spheres  in \cite{Eva11} and \cite{LLW15} to reduce the computation of $\pi_1(Symp(X, \omega))$
to $H_1(\mJ_{open};\ZZ)$ and then apply
Corollary \ref{Jopen}.





\subsection{Configurations of exceptional spheres when \texorpdfstring{$\chi(X)=5, 6, 7$}{Lg}}

We recall   the notion  of a configuration of exceptional spheres (eg.  \cite{LLW15}):

\begin{dfn}\label{configuration}
A finite collection of transversally intersected symplectic spheres $D=\bigcup D_i \subset X$ is called a (tree) configuration if
\begin{itemize}
 \item for any pair  $i,j$ with $i\ne j$, $[D_i]\ne [D_j]$ and $[D_i]\in \mathcal S_{\omega}^{-1}$;
 \item $D_i$'s are simultaneously $J$-holomorphic for some $J \in \mathcal J_{\omega}$;
\item  $D=\bigcup D_i$ is connected and its graph is a tree.
\end{itemize}
We will often use $D$ to denote such a configuration. The homological type of $D$ refers to the set of homology classes $\mfD=\{[D_i]\}$.
Further, a   configuration  is called
 {\bf standard} if the components intersect
$\omega$-orthogonally
 at every intersection point of the configuration, {\bf heavy } if $\mJ^{\mfD}=\mJ_{open}$,    and {\bf full} if $H_2(X,D;\RR)=0$.

 Denote by $\msD$  the space of configurations associated to a fixed configuration $D$, and by $\msD_0$  the subspace of standard configurations.
 \end{dfn}

We then introduce  various  subgroups of $Symp_h(X, \omega)$ associated to a fixed configuration $D$ that appear  in the diagram \eqref{summary}:
\begin{itemize}
\item  $Symp_h(X, \omega)\subset Symp(X, \omega)$ is the subgroup acting trivially on $H_2(X; \ZZ)$.
\item  $Stab(D) \subset Symp_h(X, \omega)$ is the subgroup that fixing  $D$ setwisely, but not necessarily pointwisely.
\item  $Stab^0(D)\subset Stab(D)$ is the subgroup   fixing  $D$ pointwisely.
\item   $Stab^1(D) $ is the  subgroup of $ Stab^0(D)$ fixing  $D$ pointwisely and acting trivially on the normal bundles of its components.
\item  $Symp_c(U)$ is the subgroup of $Stab^1(D)$ that have compact support in $(U, \omega|_U)$, where $U=X\setminus D$.
\end{itemize}

There are also two local groups  $Symp(D)$  and $\mG(D)$ that we will describe soon.

Recall that it suffices to assume the symplectic form is reduced since
any  symplectic form is diffeomorphic to a reduced form by Lemma \ref{reduceform} and
diffeomorphic symplectic forms have homeomorphic symplectomorphism groups.
For a reduced symplectic form on a rational surface with $\chi(X)=5, 6, 7$, consider the configurations with  the following homology types:
 \begin{itemize}
\item  ${\CC P^2} \# 2{\overline {\CC P^2}}$, $\mfD_2=\{H-E_1-E_2, E_1, E_2\}$.
\item ${\CC P^2} \#  3{\overline {\CC P^2}}$,  $\mfD_3=\{ H-E_1-E_2, H-E_2-E_3, E_1, E_2, E_3 \}$.
\item ${\CC P^2} \#  4{\overline {\CC P^2}}$,
$\mfD_4=\{ H-E_1-E_2,  H-E_3-E_4,  E_1,  E_2,  E_3, E_4\}.$
 \end{itemize}


\begin{lma} \label{l:heavy}
$\mfD_{k}$ is heavy and full.
\end{lma}

\begin{proof} $\mfD_k$ is full since it contains  a basis of $H_2(X_k;\ZZ)$.
Since $\mfD_{k}\supset \mfD'_k$, by Lemma \ref{-1open},  $\mfD_k$ is heavy.
\end{proof}

For such a configuration,
we have the following lemma about the weak homotopy type of $\msD, \msD_0$:

\begin{lma}\label{CeqJ}
For a  configuration $D$ with homology type $\mfD_k, k=2,3,4$, we have the diagram of homotopy fibrations \eqref{summary}. In particular,
both $\msD_0$ and $\msD$ are  weakly homotopic equivalent to $\mJ_{open}$.
\end{lma}

\begin{proof}
This lemma has been established in \cite{Eva11} and \cite{LLW15}  (see also \cite{LP04}),
except for the claim  that $\msD_0$ is weakly homotopic equivalent to $\mJ_{open}$.
It is shown   in  \cite{Eva11}
that $\msD_0$ is weakly homotopic to $\msD$, so it suffices to show that   $\msD$ is weakly homotopic to $\mJ_{open}$.

The crucial fact is that $\mJ_{open}(X_k, \w)=\mJ^{\mfD_k}$ for $k=2, 3, 4$ by Lemma \ref{l:heavy}. So we just need to show that
   $\mJ^{\mfD_k}$   is weakly homotopic to $\msD$.  Notice that  $ \pi: \mJ^{\mfD_k} \rightarrow \msD$
is a smooth surjective submersion (proof verbatim as in Proposition 4.8 in \cite{LP04}). Since the space of almost complex structures making a configuration pseudo-holomorphic
  is a weakly contractible subset of  $\mJ^{\mfD_k}$ (\cite{Eva11}),  the preimage of the map $ \pi$ is weakly contractible everywhere. Therefore $\pi$ is a fibration with weakly contractible fibers by \cite{Mei02}, and we have the desired weak homotopy equivalence between $\mJ^{\mfD_k}$ and $\msD$.
\end{proof}




We then analyze various groups associated with $D$ that appear in \eqref{summary}, starting from the local groups $Symp(D)$  and $\mG(D)$.

\subsubsection{ $Symp(D)$  and $\mG(D)$ }
Given a configuration of embedded
symplectic spheres $D=D_1\cup\cdots\cup D_n\subset X$, let
$I$ be the set of intersection points of the components and
 $k_i$  the cardinality of  $I\cap D_i$.

 The group $Symp(D)$ is defined to be  the direct product
$\prod_{i=1}^nSymp(D_i,I\cap D_i)$.
Here $Symp(D_i,I\cap D_i)$ denotes the group of symplectomorphisms of  $D_i\cong S^2$
fixing   the  intersection points $I\cap D_i$. 
Since  $Symp(S^2)$ acts transitively on $N$-tuples of distinct points in $S^2$ for any $N$ we can write $Symp(D_i, I\cap D_i)$ as
$Symp(S^2, k_i)$ and
\begin{equation}\label{sympc}
 Symp(D)= \prod_{i=1}^nSymp(S^2,k_i).
\end{equation}
As shown in \cite{Eva11}:
\begin{equation} \label{sympk}
Symp(S^2,1)\sim S^1;\hspace{5mm}
Symp(S^2,2)\sim S^1;\hspace{5mm}
Symp(S^2,3)\sim {\star};
\end{equation}
where $\sim$ means homotopy equivalence. In the rest of this paper, we will constantly use $\sim$ for weak homotopy equivalence.

Similarly, the  group  $\mG(D)$ (also called the symplectic gauge group) is defined to be the product $\prod_{i=1}^n\mG_{k_i}(D_i)$. Here $\mG_{k_i}(D_i)$ denotes
the group of gauge transformations of the symplectic normal bundle
to $D_i\subset X$, which are the  identity at each of the $k_i$ intersection points.
Also as shown in \cite{Eva11}:
\begin{equation} \label{gau}
\mG_0(S^2) \sim S^1 ; \hspace{5mm}  \mG_1(S^2) \sim \star; \hspace{5mm}  \mG_k(S^2) \sim \ZZ^{k-1},\ k>1.
\end{equation}
Since we assume the configuration is connected, each $k_i\geq 1$. Thus by \eqref{gau}, we have
\begin{equation}\label{mG}
 \pi_0(\mG(D)) \cong \oplus_{i=1}^n \pi_0(\mG_{k_i}(S^2))\cong \oplus_{i=1}^n  \mathbb Z^{k_i-1}.
\end{equation}

By \eqref{sympc} and \eqref{gau},  for  $D$ with homology type $\mfD_k$, we have

  \begin{itemize}
\item for  ${\CC P^2} \# 2{\overline {\CC P^2}}$, $Symp(D)\sim (S^1)^3,\mG(D)\sim \ZZ$,
\item for ${\CC P^2} \#  3{\overline {\CC P^2}}$,  $Symp(D)\sim (S^1)^5,\mG(D)\sim \ZZ^3$,
\item for ${\CC P^2} \#  4{\overline {\CC P^2}}$,
 $Symp(D)\sim (S^1)^4,\mG(D)\sim \ZZ^4$.
 \end{itemize}


\subsubsection{ $Symp_c(U), Stab^1(D),$ $Stab^0(D)$ and  $ Stab(D)$}

 Recall Proposition 3.3 in \cite{LLW15}:

\begin{lma}
 For a standard configuration $D$ with homology type $\mfD_k$,
 $Symp_c(U)\sim   \star .$
\end{lma}

Consider  the following fibration portion of the diagram \eqref{summary}:
$$Symp_c(U)\sim Stab^1(D)\to Stab^0(D)\to \mG(D).$$
It follows   that $Stab^0(D)\sim \mG(D).$


\begin{prp}\label{stab}
For a standard configuration $D$ with homology type $\mfD_k$, we have
\begin{itemize}
\item for  ${\CC P^2} \# 2{\overline {\CC P^2}}$,
$ Stab(D)  \sim\TT^2,$
\item for  ${\CC P^2} \# 3{\overline {\CC P^2}}$,
$ Stab(D)  \sim\TT^2,$
\item for ${\CC P^2} \# 4{\overline {\CC P^2}}$,
$ Stab(D)  \sim\star.$
\end{itemize}
\end{prp}

\begin{proof}
Consider 
 the following fibration portion of the diagram \eqref{summary}:
 \begin{equation}
\begin{CD}
\mG(D)\sim Stab^0(D) \to  Stab(D) \to Symp(D).
\end{CD}
\end{equation}
With the  computation of $\mG(D)$ and $Symp(D)$ above,
we have the following homotopy  fibrations for $k=2,3,4$ respectively:
$$\ZZ \to Stab(D) \to (S^1)^3,\quad
\ZZ^3 \to Stab(D) \to (S^1)^5,\quad
\ZZ^4 \to Stab(D) \to (S^1)^4.$$
We need  to understand  the connecting homomorphism
 $\pi_1(Symp(D)) \rightarrow \pi_0 (Stab^0(D))=\pi_0(\mG(D)).$
It follows from Lemma 2.9 in \cite{LLW15} that
the connecting homomorphism
 is surjective in each case.
Therefore  $Stab(D)$ is path connected and $\pi_i(Stab(D))=0$ for $i\geq 2$  in each case, and $\pi_1(Stab(D))=\ZZ^2, \ZZ^2, 0$ when $k=2,3, 4$.
Clearly,   $ Stab(D)\sim
\TT^2,
\TT^2,
\star$ for $k=2, 3,4$ respectively.
\end{proof}

We remark that, for the monotone case, $Stab(D)$
is computed in  \cite{Eva11}.



\subsubsection{$ \pi_1 (\msD_0)=\pi_1(\mJ_{open})=\ZZ^{N_{\w}}$ and the injection $\pi_1(Stab(D) )\to  \pi_1(Symp_h(X,\w))$}
Now we further move on to the last  portion of the diagram \eqref{summary}:
\begin{equation} \label{last}
  \begin{CD}
  Stab(D) @>>> Symp_h(X, \w) @>>> \msD_0.
\end{CD}
 \end{equation}
We will first  compute $ \pi_1 (\msD_0)$.


\begin{lma}\label{Abel}
For  $(X_k, \w), k=2, 3,4$, the homomorphism $\pi_1 (Symp_h(X, \w))\to  \pi_1 (\msD_0)$ is surjective and
$ \pi_1 (\msD_0)=\pi_1(\mJ_{open})$ is a free Abelian group with rank  ${N_{\omega}}$.
\end{lma}
\begin{proof}
Note that $Stab(D)$ is path connected by Proposition \ref{stab}. So  the homomorphism $\pi_1 (Symp_h(X, \w))\to  \pi_1 (\msD_0)$ associated to the homotopy fibrarion
\eqref{last} is surjective.
 Since $Symp_h(X, \w)$ is a Lie group, $\pi_1 (Symp_h(X, \w))$
 is Abelian and hence $\pi_1(\msD_0) $ is an Abelian group as well.
 Then $\pi_1(\msD_0) = H_1(\msD_0,\ZZ)$ by the Hurewicz Theorem. By Lemma \ref{CeqJ} and Corollary \ref{Jopen}, we have  $H_1(\msD_0;\ZZ) = H_1(\mJ_{open};\ZZ)=\ZZ^{N_{\w}}$.
\end{proof}


 We next analyze the associated homomorphism $g: \pi_1(Stab(D) )\to  \pi_1(Symp_h(X,\w))$.

\begin{lma}\label{injection}
The homomorphism $g: \pi_1(Stab(D) )\to  \pi_1(Symp_h(X,\w))$ is injective.
Hence we have the exact sequence
\begin{equation}\label{pi1 sequence}
0\to \pi_1(Stab(D) )\overset{g}\to  \pi_1(Symp_h(X,\w))\to \ZZ^{N_{\omega}} \to 0.
\end{equation}

\end{lma}

\begin{proof}

When $k=4,$ the injectivity is clear  since $\pi_1(Stab(D) )$ is trivial in this case.

For $\CC P^2  \# k{\overline {\CC P^2}}, k=2,3$, by Proposition \ref{stab}, $Stab(D)\sim \TT^2$.
  And the homology classes of the configuration $D$ can be realized as the boundary of the moment polytope.
  we can assume that  $D$ is a toric divisor, see also Remark \ref{toric-2}. Namely,  there is  a Hamiltonian $\TT^2$ action  on $(X, \w)$ fixing $D$. In particular, we have the  inclusions  $\TT^2\subset Stab(D)\subset Symp_h(X,\w)$.
Consider the induced map for $\pi_1$:
$$\ZZ^2  \overset{f} \rightarrow   \ZZ^2 \overset{g} \rightarrow  \pi_1 (Symp_h(X, \w)).$$
By Theorem 1.3 and Theorem 1.25 in \cite{MT10}, the homomorphism $\iota =g\circ f$ is injective.
Therefore $Im(\iota)=g(Im(f))$ is a rank 2 free Abelian group, which implies that  $g: \ZZ^2\to \pi_1 (Symp_h(X, \w))$ must be injective as well.

Finally, the exactness of the sequence \eqref{pi1 sequence} follows from Lemma \ref{Abel}.
\end{proof}

\subsection{Proof of Theorem \ref{sum} and Corollary \ref{4Q}}
We are ready to prove the following result, which includes both Theorem \ref{sum} and   Corollary \ref{4Q}. Notice that $\pi_1(Symp(X,\omega))$ is always the same as $\pi_1(Symp_h(X,\omega)),$ since  $Symp(X,\omega)$ is the extension of $Symp_h(X,\omega)$ by the homological action, which is discrete.

\begin{thm}\label{rank}
 If $(X, \omega)$ is a symplectic rational surface with  $\chi(X)\leq 7$,   then
 \begin{equation} \label{pi1}
 \pi_1(Symp_h(X,\omega))= \ZZ^{N_{\omega}} \oplus \pi_1(Symp_h(X,\omega_{mon})). \end{equation}
  $\pi_0 (Symp(X,\omega))$ and $\pi_1(Symp(X,\omega))$ are invariant on each open face of $P(X)$ and
vary between neighboring open faces.
Moreover,  if we define $Q(X)=\frac12(\chi(X)-2)(\chi(X)-3)$, then we  have the formula
\begin{equation}\label{4q}
r^+[\pi_0(Symp(X,\omega))]+  Rank [\pi_1(Symp(X,\omega))]=Q(X).
\end{equation}
Here $r^+[\pi_0(Symp(X,\omega))]$ means the number of positive roots of the Weyl type group $\pi_0(Symp(X,\omega)).$
\end{thm}
\begin{proof} There are statements in the theorem, the identity \eqref{pi1}, the stability statement, and the identity \eqref{4q}.

\bigskip

$\bullet$ We first establish  (\ref{pi1}).
For $X= {\CC P^2} \# k\overline{\CC P^2}, k=2,3,4$,
 by Lemma \ref{injection},
$$\pi_1(Symp_h(X,\omega))=  \pi_1(Stab(D)) \oplus \ZZ^{N_{\omega}}.$$
Notice that, in the monotone case, $\mJ_{open}=\mJ_{\w_{mon}}$ and so  the space $\msD_0\sim \mJ_{open} \sim  \mJ_{\w_{mon}}$ is contractible by Lemma \ref{CeqJ}.
Hence $ Symp(X,\omega_{mon})\sim Stab(D)$ from the fibration \eqref{last} and  we have  \eqref{pi1} in this case.

 For  $ {\CC P^2}, {\CC P^2} \#\overline{\CC P^2},  S^2\times S^2$,
  \eqref{pi1}  follows from  the fact that $N_{\w_{mon}}=0$  and Table \ref{0123form} below summarizing \cite{Gro85,Abr98,AM00,LP04}, where $\w_{gen}$ denotes a non-monotone
  symplectic form. Note that the $\pi_0, \pi_1$ of $\CC P^2$, monotone $S^2\times S^2$ and monotone ${\CC P^2} \#\overline{\CC P^2}$ are completely done by \cite{Gro85}; therest cases of $S^2\times S^2$ are done by \cite{Abr98} Theorem 0.1 and \cite{AM00} Theorem 1.1; and the rest cases of  ${\CC P^2} \#\overline{\CC P^2}$
are given by \cite{AM00} Theorem 1.1. The other information in the table are easy homological information.

\begin{table}[ht]
\begin{center}
\resizebox{\textwidth}{!}{\begin{tabular}{||c c c c c  c c||}
\hline\hline
X& $R(X)$  & $Symp_h(X,\w_{mon})$ & $\pi_1(Symp_h(X,\w_{mon}))$ & $\pi_1(Symp_h(X,\w_{gen}))$ & $N_{\w_{gen}}$ & $Q(X)$  \\ [0.5ex]
\hline\hline
$\CC P^2$ & $\emptyset$ & $\sim PU_3$  &$\ZZ_3$    & N/A & N/A&  0\\
\hline
${\CC P^2} \#\overline{\CC P^2},$& $ \emptyset$  &$\sim U_2$ & $\ZZ$     & $\ZZ $& 0 &$1$\\
\hline
$S^2\times S^2$& $\aA_1$ &  $\sim SO_3\times SO_3$ &$\ZZ_2\oplus \ZZ_2$ &$\ZZ\oplus \ZZ_2\oplus \ZZ_2$ & 1& $  1$\\
\hline\hline
\end{tabular}}
\caption{$\chi(X)\leq 4$}\label{0123form}
\end{center}
\end{table}



\bigskip

$\bullet$ Next,  we verify the second statement that $\pi_0 (Symp(X,\omega))$ and $\pi_1(Symp(X,\omega))$ are stable on each open face of $P(X)$ and
vary between neighboring open faces.
 For $X_3$ and $X_4$, it  follows from the discussions in Section \ref{Lagroot} (cf. Proposition \ref{MLS}) which show  that both $N_{\w}$ and $\Gamma_L$ stabilize within an open face and vary when moving to  a neighboring open face.

For the remaining case with $\chi\leq 5$,
it follows from the discussions and Figure 2  in Section \ref{chileq5}.  Notice that any  normalized reduced symplectic class on such a manifold appears in Figure \ref{2c}, which we redraw here (but with $M_1$ and
$M_2$ removed).
\begin{figure}[ht]
\begin{center}
 \begin{tikzpicture}
[scale=3.8,
axis/.style={->,blue,thick},
vector/.style={-stealth,red,very thick},
]
\coordinate (O) at (0,0,0);(-0.1,0,0)node{$O$};
\draw[axis] (0,0,0) -- (1,0,0) node[anchor=north east]{$c_1$};
\draw[axis] (0,0,0) -- (0,1,0) node[anchor=north west]{$c_2$};
\draw[red, thick]  (0.4,0.4,0)--(0,0,0) node[anchor=north ]{$c_1=c_2$};
\draw[green,thick] (0.4,0.4,0)--(0.8,0,0);
\draw[white] (0,0,0)--(0.8,0,0);
\node[draw] at(0.8,-0.15,0){$A:(1,0)$};
\node[draw] at(0.4,0.5,0){$B:(\frac12,\frac12)$};
\draw[fill=green] (0.4, 0.4, 0) circle[radius=.02 ];
\draw[fill=white] (0.8, 0, 0) circle[radius=.02 ];
\end{tikzpicture}
\caption{$P(X)$ when $\chi(X)\leq 5$ }
\end{center}
\end{figure}

The statement is trivial for $X_0= {\CC P^2}$ since $O$  is the only normalized reduced class.

For $X_1={\CC P^2} \#\overline{\CC P^2},$ $\pi_0(Symp(X_1,\w))=\{1\}$ and $\pi_1(Symp(X_1,\w))=\ZZ$   on the 1-dimensional  open face $OA$,   which is   the whole cone $P_1$.

For $X_2={\CC P^2} \#2\overline{\CC P^2},$ the cone $P_2$  is the disjoint union of the two open faces $OB$ and $OAB$.  On the 1-dimensional  open face $OB$, $\pi_0(Symp(X_2,\w))=\ZZ_2$ and $\pi_1(Symp(X_2,\w))=\ZZ^2$.
On the open face $OAB$,  $\pi_0(Symp(X_2, \w))=\{1\}$ and $\pi_1(Symp(X_2,\w))=\ZZ$.

For $S^2\times S^2,$ the cone is the half closed and half open interval $[B, A)$, where $B$ is the monotone class. So the conclusion can be read off from Table \ref{0123form} above.

\bigskip

$\bullet$ Finally,  we establish  \eqref{4q}.
Note  that the number of Lagrangian spherical classes  $N_{\w,L}$ is the same as $r^+[\pi_0(Symp(X,\omega))]$.
 And  by  equation \eqref{SL'} which is valid   when $\chi(X)\leq 7,$ we have   $N_{\w}+ r^+[\pi_0(Symp(X,\omega))] = |R^+(X)|$.   Then \eqref{4q}
  follows from the previously established \eqref{pi1} and the formula
  \begin{equation}\label{mon-R}
  Rank [\pi_1(Symp(X,\omega_{mon}))]+ |R^+(X)|=\frac{1}{2}(\chi(X)-2)(\chi(X)-3)
  \end{equation}  when $\chi(X)\leq 7$.
  The  formula \eqref{mon-R} can be directly verified case by case.  When  $\chi(X)\leq 4$,  it follows  from Table \ref{0123form} above.  When $\chi(X)=5, 6, 7$, we have
  $|R^+(X)|=1, 4, 10$ and $  Rank [\pi_1(Symp(X,\omega_{mon}))]=2, 2, 0$  respectively, and hence \eqref{mon-R} also holds in these cases.
\end{proof}

\begin{rmk}\label{s1gen}
In \cite{AP13}, the computation of $\pi_1(Symp_h(X,\w)$ for any given form on
3-fold blow up of ${\CC P^2}$ is given.  There the strategy is counting torus
(or circle) actions. And a generating set of $\pi_1(Symp_h(X,\w)$ is given using
circle action. Note that our approach gives another (minimal) set of $\pi_1(Symp_h(X,\w)$.
We give the correspondence of the two generating sets:
 By Remark \ref{toric-2}, any -2 symplectic sphere in 3 fold blow up of ${\CC P^2}$, there
is a semi-free circle $\tau$ action having this -2 symplectic sphere as fixing locus, where $\tau$ is a generator of $\pi_1(Symp_h(X,\w)$.

\end{rmk}

\subsubsection{Some remarks}

When $\chi(X)\geq 8$, due to the fact that $\pi_0(Symp_h(X,\w))$ could be non-trivial, we consider the quantity
$$r^+[\pi_0(Symp(X,\omega))]+  Rank [ \pi_1(Symp_h(X,\omega))]-Rank [\pi_0(Symp_h(X,\omega))],$$
which is conjectured to be a pure topological quantity $Q(X).$
Note that this quantity $Q(X)$ is $\frac12(\chi(X)-2)(\chi(X)-3)$,  which is the same as in Theorem \ref{sum} and   Corollary \ref{4Q} when $\chi(X)\leq7.$
We  prove in \cite{LLW16} that when $\chi(X)=8$ this is ture, and $Q(X)=15$.

In \cite{McD08} Corollary 6.9, McDuff gave a blowup approach to obtain an upper bound of the rank
$\pi_1(Symp(X,\omega))$  for  a symplectic rational 4 manifold $(X, \w)$.
We can show that our Theorem \ref{rank} actually attained the upper bound via her approach
when $\chi(X)\leq 7$.
 In the follow-up paper  \cite{LLW16}, we will explore a generalization of  Corollary 6.9 in \cite{McD08} to estimate the rank of
$\pi_1(Symp(X_5,\omega))$ where $X_5$ is the 5-point blowup of the projective plane. Although the upper bound is not
always optimal, it is crucial for the rank computation.
As an potential application of computing the rank of $\pi_1(Symp)$, one can follow \cite{Pol98} to compare the rank with the number of Hamiltonian isotopy classes of monotone Lagrangian torus  to approach the Hofer diameter conjecture, which says that the diameter of $Ham(X,\w)$ with respect to the canonical bi-invariant Finsler metric (called Hofer metric) is infinity.

There is also a general lower bound of the rank of $\pi_1(Symp(X,\w))$ given by counting circle actions in \cite{Ped15}.
In \cite{AP13},   when $\chi(X)=6$, a generating set  of $\pi_1(Symp(X,\w))$ as an Abelian group is given in terms of circle actions,
and this set is shown to also generate  the rational homotopy Lie algebra of  $Symp(X,\w)$.
Notice that any symplectic form on a  rational 4 manifold $X$ with  $\chi(X) \leq 6$ is a toric from.
Our approach gives another (minimal) generating set of $\pi_1(Symp(X,\w))$ in terms of  $(-2)$-symplectic spheres.
When $\chi(X)=6$,  the two generating sets correspond via
 Remark \ref{toric-2}: any $(-2)$-symplectic sphere in $(X, \w)$ with $\chi(X)=6$ arises as a fixed point component of
a Hamiltonian  circle action.


\begin{rmk}\label{anj}
Anjos-Eden (\cite{Anjos, AE17}) further studied some toric cases for the  4-fold blow up of ${\CC P^2}$, including the open face $M_4OBC$  in Table \ref{4form}. In particular, they
compute the rank of $\pi_1(Symp(X,\w))\otimes \QQ$ in this case and further   prove that the group  generators of $\pi_1(Symp(X,\w))\otimes \QQ$  generate the entire rational homotopy Lie algebra of $Symp(X,\w)$.


\end{rmk}

\subsubsection{Explicit computation of  \texorpdfstring{$\pi_1(Symp_h(X,\w))$  for $X={\CC P^2} \# 4\overline{\CC P^2}$}{Lg}}
Finally we apply  Theorem \ref{rank} to explicitly compute $\pi_1(Symp_h(X,\w))$ on each  open face of the normalized reduced cone $P_4$.
Recall that by Proposition \ref{nrsc} and \ref{MLS},
the normalized reduced cone of  ${\CC P^2} \# 4\overline{\CC P^2}$ is convexly generated by 4 rays $\{M_4O,M_4A,M_4B,M_4C\}$, with  $M_4=(\frac13,\frac13,\frac13,\frac13),$ $O=(0,0,0,0),$ $A=(1,0,0,0),$
$B=(\frac12,\frac12,0,0),$ $C=(\frac13,\frac13,\frac13,0)$.
The results are listed in the following  table. Note that the result on any open face contains the letter ``A" is new (the other cases are also obtained  by \cite{AE17}). Recall that in the form, $\lambda=c_1+c_2+c_3,$ as in section \ref{3to8}.


\begin{table}[ht]
\begin{center}
\resizebox{\textwidth}{!}{\begin{tabular}{||c c c c c  ||}
\hline\hline
open face& $\Gamma_L$  & ${N_{\omega}}$ & $\pi_1(Symp_h(X,\w))$ &$\omega$ area  \\ [0.5ex]
\hline\hline
Point $M_4$& $\aA_4$ & 0& trivial &   monotone, $ \lambda=1; c_1=c_2=c_3=c_4 $\\
\hline
 $M_4$O& $\aA_3$ & 4& $\ZZ^4$& $  \lambda<1; c_1=c_2=c_3=c_4 $\\
\hline
$M_4$A& $\aA_3$ & 4& $\ZZ^4$&  $\lambda=1;c_1>c_2=c_3=c_4 $\\
\hline
$M_4$B&$\aA_1\times \aA_2$  &6& $\ZZ^6$& $\lambda=1;c_1=c_2>c_3=c_4 $ \\
\hline
$M_4$C&$\aA_1\times \aA_2$   &6& $\ZZ^6$& $\lambda=1;c_1=c_2=c_3>c_4 $ \\
\hline
$M_4$OA&$\aA_2$  &7& $\ZZ^7$& $ \lambda<1;  c_1>c_2=c_3=c_4$\\
\hline
$M_4$OB& $\aA_1\times \aA_1$ &8& $\ZZ^8$& $ \lambda<1; c_1=c_2>c_3=c_4$ \\
\hline
$M_4$OC&$\aA_2$   &7& $\ZZ^7$&$ \lambda<1; c_1=c_2=c_3>c_4 $\\
\hline
$M_4$AB&$\aA_2$  &7& $\ZZ^7$&  $ \lambda=1;c_1>c_2>c_3=c_4$\\
\hline
$M_4$AC&$\aA_1\times \aA_1$ &8& $\ZZ^7$&   $ \lambda=1;c_1>c_2=c_3>c_4  $\\
\hline
$M_4$BC&$\aA_1\times \aA_1$ &8& $\ZZ^7$&  $ \lambda=1;c_1=c_2>c_3>c_4 $  \\
\hline
$M_4$OAB& $\aA_1$ &9& $\ZZ^8$& $\lambda<1; c_1>c_2>c_3=c_4$\\
\hline
$M_4$OAC& $\aA_1$    &9& $\ZZ^9$& $ \lambda<1;   c_1>c_2=c_3>c_4  $\\
\hline
$M_4$OBC& $\aA_1$   &9& $\ZZ^9$&  $ \lambda<1; c_1=c_2>c_3>c_4 $\\
\hline
$M_4$ABC& $\aA_1$   &9& $\ZZ^9$&  $ \lambda=1;   c_1>c_2>c_3>c_4  $\\
\hline
$M_4$OABC& trivial   &10& $\ZZ^{10}$&  $ \lambda<1;c_1>c_2>c_3>c_4 $\\
\hline\hline
\end{tabular}}
\caption{$\Gamma_L$ and  $\pi_1(Symp_h(X,\w))$  for  ${\CC P^2} \# 4\overline{\CC P^2}$}\label{4form}
\end{center}
\end{table}

\newpage

\printbibliography

\end{document}